\documentclass[aos]{imsart}
\RequirePackage[OT1]{fontenc}
\RequirePackage{amsthm,amsmath,amssymb}
\usepackage{epsfig}
\usepackage{color}
\usepackage[round]{natbib}
\usepackage{url} 

\startlocaldefs

\usepackage{amsthm,amsfonts,amssymb, bm, mathtools}

\usepackage{multirow}
\usepackage{array}
\usepackage{comment, float, booktabs}

\DeclarePairedDelimiter\ceil{\lceil}{\rceil}

\DeclareMathOperator*{\argmax}{argmax}
\DeclareMathOperator*{\argmin}{argmin} 
\def\0{\mathbf{0}}

\def\hD{\hat{D}}
\def\Ebb{\mathbb{E}}
\def\cF{\mathcal{F}}
\def\hG{\widehat{G}}
\def\tG{\widetilde{G}}
\def\cG{\mathcal{G}}
\def\cH{\mathcal{H}}
\def\bI{\mathbf{I}}

\def\cL{\mathcal{L}}

\def\real{\mathbb{R}}

\def\cX{\mathcal{X}}
\def\cY{\mathcal{Y}}

\def\vtheta{{\bm{\theta}}}

\def\vphi{{\bm{\phi}}}

\newtheorem{theorem}{Theorem}[section]
\newtheorem{lemma}{Lemma}[section]

\newtheorem{lemma*}{Lemma}

\newtheorem{assumption}{Assumption}
\numberwithin{equation}{section}

\endlocaldefs

\begin{document}
\begin{frontmatter}

\title{Deep Generative Survival Analysis: Nonparametric Estimation of Conditional Survival  Function}
\runtitle{}
\begin{aug}
  \author[A]{\fnms{Xingyu}  \snm{Zhou}\thanksref{t1}\ead[label=e1]{xingyu-zhou@uiowa.edu}},
  \author[B]{\fnms{Wen} \snm{Su}\thanksref{t1}\ead[label=e2]{jenna.wen.su@connect.hku.hk}},
  \author[C]{\fnms{Changyu}  \snm{Liu}\thanksref{t1}\ead[label=e3]{cy.u.liu@connect.polyu.hk}},
  \author[D]{\fnms{Yuling}  \snm{Jiao}\ead[label=e4]{yulingjiaomath@whu.edu.cn}}, \\
  \author[C]{\fnms{Xingqiu}  \snm{Zhao}\ead[label=e5]{xingqiu.zhao@polyu.edu.hk}} \and
  \author[A]{\fnms{Jian}  \snm{Huang}\ead[label=e6]{jian-huang@uiowa.edu}}

\thankstext{t1}{Equal contribution}

  \address[A]{Department of Statistics and Actuarial Science, University of Iowa, Iowa City, Iowa 52242, USA\\
   \printead{e1,e6}}

  \address[B]{Department of Statistics and Actuarial Science, The University of Hong Kong, Hong Kong, China \\
   \printead{e2}}

\address[C]{Department of Applied Mathematics,
   The Hong Kong Polytechnic University,
   Hong Kong, China \\
 \printead{e3,e5}}

  \address[D]{School of Mathematics and Statistics, Wuhan University, Wuhan, Hubei, 430072, China\\
          \printead{e4}}

\end{aug}

\begin{abstract}

We propose a deep generative approach to nonparametric estimation of conditional survival and hazard functions with right-censored data. The key idea of the proposed method is to first learn a conditional generator for the joint conditional distribution of the observed time and censoring indicator given the covariates, and then construct the Kaplan-Meier and Nelson-Aalen estimators based on this conditional generator for the conditional hazard and survival functions. Our method combines ideas from the recently developed deep generative learning and classical nonparametric estimation in survival analysis. We analyze the convergence properties of the proposed method and establish the consistency of the generative nonparametric estimators of the conditional survival and hazard functions. Our numerical experiments validate the proposed method and demonstrate its superior performance in a range of simulated models. We also illustrate the applications of the proposed method in constructing prediction intervals for survival times with the PBC (Primary Biliary Cholangitis) and SUPPORT (Study to Understand Prognoses and Preferences for Outcomes and Risks of Treatments) datasets.
\end{abstract}



\end{frontmatter}

\maketitle

\section{Introduction}
\label{intro}
Censored survival data arise in many fields of scientific research, including biomedical studies,  epidemiology, and econometrics, among others. Therefore, it is of great interest to develop methods that can effectively analyze such data.
In this paper, we propose a novel deep generative approach to  nonparametric estimation of conditional survival and hazard functions with right-censored data.
Our proposed approach combines  ideas from the recently developed deep generative learning  and classical nonparametric survival analysis methods, and  leverages the
power of neural network function for approximating multivariate functions.

There is a vast literature on the analysis of censored survival data. A majority of the existing methods
relies on certain structural and functional assumptions on the conditional hazards or conditional distributions of survival time. In particular,
many important
methods assume a semiparametric regression model for survival time
and develop techniques for dealing with difficulties due to censoring.
For instance, the widely-used
proportional hazards model \citep{cox1972}
assumed that the conditional hazard function for survival time $T$ of a subject with covariate $X$ takes the form
$
\lambda(t|X)=\lambda_0(t) \exp(X'\beta),
$
where $\lambda_0$ is an unspecified baseline hazard function and $\beta$
is a vector of regression coefficients for the effect of $X$ on the conditional hazard.
This is a semiparametric model since the functional form of $\lambda_0$ is unspecified.
To avoid the difficulty of nonparametric estimation of $\lambda_0$,
\citet{cox1972}
also proposed a conditional likelihood without the baseline hazard $\lambda_0$ for estimating $\beta$.
Subsequently, \citet{cox1975} introduced the celebrated partial likelihood, generalizing the ideas of conditional and marginal likelihoods. An appealing feature of the partial likelihood is that it is only a function of
$\beta$ without involving $\lambda_0$ and can be used for estimation of $\beta$
and making inference just like a usual likelihood.

\citet{ag1982} considered a counting process formulation of the Cox model and studied the large sample properties of the partial likelihood estimator using the martingale theory.
Additionally, the additive hazards model \citep{
co1984, ly1994, ms1994} and the accelerated failure time model \citep{
bj1979, tsiatis1990, wyl1990} have been proposed as
alternatives to the Cox model. For more details on these models, we refer to \citet{fh1991} and \citet{kp2002}.


More recently, there have been some interesting works on applications of deep learning in survival analysis.
For example,
\citet{chapfuwa2018} proposed a deep adversarial learning approach to nonparametric estimation for time-to-event analysis, which adopted the GAN objective function
\citep{goodfellow2014generative}
to  exploit information from censored observations;
\citet{ZMW2021a, ZMW2021b} developed neural network approaches for a general class of hazard models, and also considered using neural network functions for approximation in a partly linear Cox model and established the asymptotic properties of the partial likelihood estimator.

Our proposed approach was inspired by the {generative adversarial networks}
(GAN) \citep{goodfellow2014generative} and the Wasserstein GAN (WGAN) \citep{arjovsky17}, which were developed to learn high-dimensional unconditional distributions nonparametrically.
Several studies have generalized GANs to conditional distribution learning in the complete data setting. For instance, \citet{mirza2014cgan} proposed the conditional generative adversarial networks (cGAN), which solves a two-player minimax game using an objective function with the same form as that of the original GAN.
\cite{zjlh2021} proposed a generative approach to conditional sampling based on the noise-outsourcing lemma and distribution matching, where the Kullback-Liebler divergence was used for matching the generator distribution and the data distribution. The authors also established consistency of the conditional sampler with respect to the total variation distance. 
\cite{lzjh2021} studied generative conditional learning
using the Wasserstein distance, where
the non-asymptotic error bounds and convergence rates for the conditional sampling distribution was established.



Building upon the aforementioned works on generative learning,
we propose a model-free,
deep generative approach to analyzing right censored survival data and refer to it as the Generative Conditional Survival function Estimator (GCSE). The key idea and novelty of GCSE is that it first learns a conditional generator
for the conditional joint distribution of the observed time and censoring indicator given covariates, then construct the Kaplan-Meier and Nelson-Aalen estimators using samples from this conditional generator for the conditional hazards and survival functions.
Specifically, let $T$ be the survival time and $C$ be the censoring time, hence we observe $(Y, \Delta)$, where $Y=\min(T,C)$, $\Delta=I(T\le C)$, and $I(\cdot)$ is an indicator function. Let $X \in \mathbb{R}^d$ denote a vector of covariates. We also include a random vector $\eta \in \mathbb{R}^q$ with distribution $P_{\eta}$, which is generated independent of data and used as the source of randomness for generating from the target conditional survival function.

We first estimate a function $G: \mathbb{R}^q \times \mathbb{R}^d \to \mathbb{R}^+\times \{0, 1\}$ nonparametrically such that $G(\eta, x)$ follows the conditional distribution of $(Y, \Delta)$ given $X=x$.  Then, to sample from the conditional distribution, we only need to calculate $G(\eta, x)$ after generating $\eta$ from the reference distribution. Computationally, we will use neural network functions to approximate $G.$ 
The main task in constructing GCSE is to estimate the conditional generator $G$ for the conditional distribution of $(Y, \Delta)$ given $X$.
The conditional generator is learned by matching the joint distribution involving the conditional generator and the predictor as well as the joint distribution of the response and the predictor.
The Wasserstein distance
serves as the discrepancy measure for matching the joint distributions. The validity of this approach is guaranteed by the conditional independence assumption of survival and censoring time given covariates. Clearly, GCSE is different from the existing methods in survival analysis, where the general strategy is to focus on modeling the relationship between the unobserved underlying survival time and the covariates and then deal with the complication due to censoring.

With the formulation described above, we have transformed the problem of estimating the conditional survival function to a nonparametric generative conditional learning problem.
An attractive aspect of GCSE is that it does not make any parametric assumptions on the
structure of the conditional hazard function, which distinguishes it from the existing semiparametric methods for censored survival data.
The conditional independence assumption is needed for identifying the conditional distribution of the survival time given the covariates, which is commonly assumed in the existing literature.
GCSE leverages the recent developments in deep generative learning and apply it
to censored survival data analysis.
It also benefits from the recent advances in computational platforms and algorithms such as tensor flow \citep{abadi2016tensorflow} and stochastic gradient descent \citep{kingma2015adam} for solving complex optimization problems.


The paper is organized as follows. In Section 2, we introduce the GCSE method and provide comprehensive details on the two-step estimation procedure. We establish the asymptotic consistency in Section 3 with detailed proofs relegated to Section 7 and describe the specifics of implementation steps in Section 4. In Sections 5, we conduct extensive simulation studies  and present applications to the PBC and SUPPORT datasets. Section 6 summarizes the proposed methodology and potential future research opportunities.
Additional numerical results are given in the Appendix.

\section{Generative conditional survival function estimator} 
Let $T\in \mathbb{R}^+$ be a survival time and $X \in \mathbb{R}^d$ be a $d$-dimensional predictor.  In the presence of random censoring, we do not observe $T$ directly, but only observe
$Y=\min\{T, C\}$ and $\Delta=1\{T\le C\}$, where $C$ is a random censoring time.
Together with the predictor $X$, the observable variable is $(Y, \Delta, X)$.
We are
interested in learning the conditional distribution
of the survival time $T$ given
the covariate $X$ nonparametrically.
However, since only a censored version of $T$ is observable, we cannot directly apply the existing generative learning methods to learn this conditional distribution.
The key idea of GCSE is to first learn a conditional generator for the conditional distribution of the
observed event time and the censoring indicator given the covariate,  and then estimate the conditional hazard and survival functions using the existing nonparametric methods for censored data.
Specifically, GCSE is a two-step approach consisting of the following steps.

\begin{itemize}
\item Step 1. We learn a conditional generator nonparametrically for the conditional distribution $P_{(Y, \Delta)\mid X}.$  Let $\eta \in \mathbb{R}^s$ be a random variable from a simple reference distribution such as normal or uniform, where $s \ge 1.$ We find a $G=(G_1, G_2)$, where $G_1: \mathbb{R}^s \times \cX \mapsto \mathbb{R}^+$ and
     $G_2: \mathbb{R}^s \times \cX \mapsto \{0, 1\},$ such that
\begin{align}
\label{Gdef}
(G_1(\eta, x), G_2(\eta, x))\sim P_{(Y, \Delta)\mid X=x}, \ x \in \cX.
\end{align}
That is, $(G_1(\eta, x), G_2(\eta, x))$ is distributed as the observed time $Y$ and its censoring indicator $\Delta$ for an individual whose covaraite $X=x$.

\item Step 2. We apply the Kaplan-Meier and Nelson-Aalen estimators to the samples from the
conditional generator $(G_1(\eta, x), G_2(\eta, x))$ for any given $x \in \cX$. This yields the desired conditional hazard and
survival functions of $T$ given $X=x$.
\end{itemize}

In Step 1, the existence of such a $G$ is guaranteed by the noise outsourcing lemma
 (Theorem 5.10 in \cite{kall2002}).
To ensure that  Step 2 leads to a consistent estimator, we make the usual conditional independence assumption:
\begin{assumption}
\label{asp1}
The failure time $T$ and the censoring time $C$ are conditionally independent given the covariate $X$.
\end{assumption}

Below we describe Steps 1 and 2 in detail.

\subsection{Estimation of the conditional generator}
As described above, the key idea in our proposed method is to learn a conditional generator for the joint conditional distribution $P_{(Y, \Delta)\mid X}$ of the censored response and the censoring indicator given the covariate.
By the noise outsourceing lemma  (Theorem 5.10 in
\cite{kall2002}), there exist a function $G=(G_1, G_2): \mathbb{R}^s \times \mathbb{R}^d \to \mathbb{R}^+ \times \{0,1\}$ and a random variable $\eta \sim P_{\eta}$  in $\mathbb{R}^s$  independent of $X$ such that
 \begin{align}
 \label{NO1}
 (X, G_1(\eta, X), G_2(\eta, X)) = (X, Y, \Delta), \ \text{ almost surely}.
 \end{align}
 Denote the joint distribution of $(X, G_1(\eta, X), G_2(\eta, x))$ by $P_{X,G_1, G_2}.$
 Since $\eta$ and $X$ are independent, $(G_1(\eta, x), G_2(\eta, x)) \sim P_{(Y, \Delta)\mid X=x}$ if and only if (\ref{NO1}) holds.
 Therefore,  to find the conditional generator, it is equivalent to find a $(G_1^*, G_2^*)$
 such that the joint distribution $P_{X,G_1^*, G_2^*}$
 of $(X, G_1^*(\eta, X), G_2^*(\eta, X))$ matches the joint distribution $P_{X, Y, \Delta}$ of $(X, Y, \Delta)$.


Let $\mathcal{D}$ be a divergence measure between the distributions $P_{X, G_1, G_2}$
and $P_{X, Y, \Delta}.$
Suppose $\mathcal{D}$ has the following two properties
\begin{itemize}
\item  $\mathcal{D}(P_{X,G_1. G_2}, P_{X,Y, \Delta}) \ge 0$ for every measurable $(G_1, G_2),$ 
\item
 $\mathcal{D}(P_{X,G_1, G_2}, P_{X, Y, \Delta})=0$
if and only if $P_{X, G_1, G_2}=P_{X, Y, \Delta}.$
\end{itemize}
Then we can characterize $(G_1^*, G_2^*)$ as a solution to the minimization problem
\[
(G_1^*, G_2^*)\in \argmin_{G_1, G_2} \mathcal{D}(P_{X, G_1, G_2}, P_{X, Y, \Delta}).
\]

Many divergence measures for probability distributions have been proposed in the literature.
In this work, we take $\mathcal{D}$ to be  the 1-Wasserstein distance  \citep{villani2008optimal2}, which has been used in the context of generative learning  \citep{arjovsky17}.
A computationally convenient form of the 1-Wasserstein metric is the Monge-Rubinstein dual \citep{villani2008optimal2},
\begin{equation}
\label{1wassersteinb}
d_{W_1}(\mu, \nu)=\sup_{f \in \cF^1_{\text{Lip}}}
\left\{\Ebb_{U \sim \mu} f(U) - \Ebb_{V\sim \nu} f(V) \right\}.
\end{equation}
where $\cF^1_{\text{Lip}}$ is the $1$-Lipschitz class,
\begin{equation}
\label{Lpi1}
\cF^1_{\text{Lip}}=\{f: \mathbb{R}^k \to \mathbb{R}, |f(u)-f(v)| \le \|u-v\|_2,\  u \text{ and } v \in \mathbb{R}^k\}.
\end{equation}

An attractive property of the Wasserstein distance is that it metricizes the space of probability distributions under mild conditions. In comparison, other discrepancy measures including the Kullback-Liebler and Jensen-Shannon divergences do not have this property.  Also, since the computation of the Wasserstein distance does not involve density functions, we can use it to learn distributions without a density function, such as distributions supported on a set with a lower intrinsic dimension than the ambient dimension.

By (\ref{1wassersteinb}), the 1-Wasserstein distance between $P_{X, G_1, G_2}$ and $P_{X, Y, \Delta}$ is
\begin{align*}
d_{W_1}(P_{X, G_1, G_2}, P_{X, Y, \Delta})
 = \sup_{D\in \cF^1_{\text{Lip}}}\left\{\Ebb_{(X, \eta)}
D(X,G_1(\eta,X), G_2(\eta, X))-\Ebb_{(X, Y,\Delta)}
D(X,Y, \Delta)\right\}.
\end{align*}
We have $d_{W_1} (P_{X, G_1, G_2},  P_{X, Y, \Delta}) \ge 0$ for every measurable $G$ and
$d_{W_1}(P_{X, G_1, G_2}, P_{X, Y, \Delta}) =0$ if and only if $P_{X, G_1, G_2}=P_{X, Y, \Delta}.$
So a sufficient and necessary condition for
\[
(G_1^*, G_2^*) \in \argmin_{G_1, G_2}  d_{W_1}(P_{X, G_1, G_2}, P_{X, Y, \Delta})
\]
is  $P_{X, G_1^*, G_2^*} = P_{X, Y, \Delta}$, which implies
 $(G_1^*(\eta, x), G_2^*(\eta, x)) \sim P_{(Y, \Delta)\mid X=x},$ for any $x \in \cX.$
Therefore, at the population level, the problem of finding the conditional generator can be
formulated as the minimax problem:
\[
\argmin_{(G_1, G_2) \in \cG} \argmax_{D \in \cF^1_{\text{Lip}}} \cL (G_1, G_2, D),
\]
where
\begin{equation}
\label{objL}
\cL(G_1, G_2, D)=\Ebb D(X,G_1(\eta,X), G_2(\eta, X))-\Ebb D(X,Y, \Delta).
\end{equation}

Suppose we have a random sample of censored observations  $\{(X_i, Y_i, \Delta_i), i=1, \ldots, n\}$
that are i.i.d. $(X, Y, \Delta).$
 Let $\{\eta_i, i=1,\ldots,n\}$ be random variables independently generated from  $P_{\eta}$.
An empirical version of $\cL(G_1,G_2, D)$ based on  $(X_i, Y_i, \Delta_i)$
and $\eta_i, i=1,\ldots, n,$ is
\begin{align}
\label{objE}
\cL_n(G,D) =& \frac{1}{n}\sum_{i=1}^n D(X_i, G_1(\eta_i,X_i), G_2(\eta_i, X_i))
-\frac{1}{n}\sum_{i=1}^n D(X_i,Y_i, \Delta_i) .
\end{align}

We use a feedforward neural network $(G_{\vtheta_1}, G_{\vtheta_2})$ with parameter $(\vtheta_1, \vtheta_2)$ for estimating the conditional generator
and a second network $D_{\vphi}$ with parameter $\vphi$ for estimating the discriminator.
{\color{black}
We use sigmoid activation at the output layer of $G_{\vtheta_2}$ so that it takes values in
$[0. 1].$
The network parameters  $(\vtheta_1, \vtheta_2)$ and $\vphi$ are estimated
by solving the minimax problem:
\begin{equation}
(\hat{\vtheta}_{1n},\hat{\vtheta}_{2n}, \hat{\vphi}_n)=\argmin_{\vtheta_1, \vtheta_2}\argmax_{\vphi}\cL_n(G_{\vtheta_1}, G_{\vtheta_2}, D_{\vphi}).
\label{2}
\end{equation}
Denote $\tG_n = (G_{\hat{\vtheta}_{1n}}, G_{\hat{\vtheta}_{2n}}).$
Since the censoring indicator $\Delta \in \{0, 1\}$, we take the  estimated conditional generator to be
\begin{equation}
\label{3}
\hG_n = (\hG_{1n}, \hG_{2n})=(G_{\hat{\vtheta}_{1n}}, I\{ G_{\hat{\vtheta}_{2n}}\geq 0.5 \}),
\end{equation}
and the estimated discriminator is $\hD_n=D_{\hat{\vphi}_n}$.
Therefore, in our implementation, we determine  whether {\color{black}$G_{\hat\vtheta_{1n}}$ } is censored based on if the value of {\color{black}$G_{\hat\vtheta_{2n}}$}  is smaller than 0.5.
}


\subsection{The conditional Kaplan-Meier and Nelson-Aalen estimators}
We first describe the conditional cumulative hazard and survival functions based on the conditional generator $G$ at the population level.
To use the standard notation in survival analysis, we write
\begin{align}
\label{YDdef}
(Y^x, \Delta^x)=(G_1(\eta, x), G_2(\eta, x)) \sim P_{(Y, \Delta)\mid x}, \quad x \in \cX.
\end{align}
That is, $(Y^x, \Delta^x)$ is distributed as the conditional distribution of $(Y, \Delta)$ given $X=x$.
Let $H^x(t)=H(t\mid x)$ and $H^x_1(t)=H_1(t\mid x)$ be the conditional  distribution of $Y^x$ and the conditional subdistribution of $(Y^x, \Delta^x=1)$, respectively, that is,
\begin{align}
\label{Hdef1}
1-H^x(t)=\mathbb{P}(Y^x > t), \
H^x_1(t)=\mathbb{P}(Y^x \le t, \Delta^x=1).
\end{align}
Under the conditional independence Assumption \ref{asp1}, we have
\begin{align}
\label{ind1}
1-H^x(t)=(1-F^x(t))(1-Q^x(t)) \ \text{ and } \
H_1^x(t) = \int_{[0,t]}(1-Q_{-}^x(u)) dF^x(u),
\end{align}
where $Q^x$ is the conditional distribution function of the censoring variable $C$ given $X=x$, and $Q_{-}^x(t)$ is the left limit of $Q^x$ at point $t$.
The conditional cumulative hazard function of the \textit{survival time $T$ given $X=x$}  is
\begin{align}
\label{cH1}
\Lambda^x(t)\equiv \Lambda(t\mid x) = \int_{[0, t]} \frac{1}{1-F^x_{-}(u)} dF^x(u).
\end{align}
Multiplying $1-Q_{-}^x(u)$ in the numerator and denominator of the integrant in (\ref{cH1}), expression (\ref{ind1}) leads to the familiar representation of the cumulative conditional hazard function $\Lambda^x$ in terms of $H^x$ and $H^x_1$,
\begin{align}
\label{cH2}
\Lambda^x(t)=
\int_{[0,t]} \frac{1}{1-H^x_{-}(u)} dH^x_1(u).
\end{align}
Therefore,  $\Lambda^x$ is identifiable through $H^x$ and $H^x_1$,
due to the conditional independence Assumption \ref{asp1}.
Then the conditional survival function $S^x(t)\equiv 1- F^x(t)$ can be obtained via the product integral formula
\begin{align}
\label{KM1}
S^x(t) = \prod_{u\in [0, t]} (1-\Lambda^x\{u\}) {\exp(-\Lambda^{xc}(t))},
\end{align}
where $\Lambda^x\{u\}$ is the jump of $\Lambda^x$ at $u$ and $\Lambda^{xc}$ is the continuous
part of $\Lambda^x,$ see, for example,
Lemma 25.74 in \citet{vaart2000asymptotic}.

We estimate the conditional cumulative hazard and survival functions as follows.
Let $(\widehat G_{1n}, \widehat G_{2n})$ be the estimated conditional generator.
For any $x\in \cX$, the distribution of $(\widehat G_{1n}(\eta, x), \widehat G_{2n}(\eta, x))$ is an estimator of the conditional distribution $P_{(T,\Delta)\mid X=x}.$
Similarly to (\ref{YDdef}),
we write $ (\widehat Y^x, \hat \Delta^x)=(\widehat G_{1n}(\eta,x), \widehat G_{2n}(\eta, x)).$
Let $\widehat H^x_n$ and $\widehat H^x_{1n}$ be the distribution function of $\widehat Y^x$ and
$(\widehat Y^x, \hat \delta^x=1)$, respectively, that is,
\begin{align}
\label{hHn}
1-\widehat H^x_n(t) = P(\widehat Y^x > t), \
\widehat H^x_{1n}(t)=P(\widehat Y^x \le t, \hat\Delta^x=1).
\end{align}
These are the counterparts of  $H^x$ and $H^x_1$ defined in
(\ref{Hdef1}), but with respect to the estimated conditional distribution
$\widehat P_{(Y, \Delta)|X=x}$ of $(Y,\Delta)$ given $X=x$.
Then by (\ref{cH2}) and (\ref{KM1}), we can estimate the conditional hazard and survival functions by
\begin{align}
\label{cH2hat}
\widehat \Lambda^x_n(t)=
\int_{[0,t]} \frac{1}{1-\widehat H^x_{n-}(u)} d\widehat H^x_{1n}(u),
\end{align}
and
\begin{align}
\label{KM1hat}
\widehat S^x_n(t) = \prod_{u\in [0, t]} (1-\widehat \Lambda^x_n\{u\}) {\exp(-\widehat\Lambda^{xc}_n(t))}.
\end{align}
Although  the estimators in (\ref{cH2hat}) and (\ref{KM1hat}) do not have an analytical expression,
we can use $\widehat G_n(\cdot, x)$ to generate samples that are
distributed as  $\widehat P_{(Y, \Delta)|X=x}$ and obtain
Monte Carlo approximations of $\widehat{\Lambda}^x_n$ and $\widehat S^x_n$, as described below.

We first generate a random sample $\{\eta_j, j=1 \ldots, m\}$ that are i.i.d.  $P_{\eta}$ and compute
$\{\widehat G_n(\eta_j, x),  j=1,\ldots, m\}$ for a given $x \in \cX$, where $m \ge 1$ is
a positive integer.
We can set  $m$ as large as we like to achieve the desired precision but
within the available computing capability.
Denote $(\widehat Y^x_{nj}, \widehat \delta^x_{nj})=\widehat G_n(\eta_j, x), j=1, \ldots, m.$
The empirical version of
(\ref{hHn}) is
\begin{align}
\label{ind1e}
1-\widehat H^x_{nm}(t) = \frac{1}{m}\sum_{j=1}^m I( \widehat Y^x_{nj} > t), \
\widehat H^x_{1nm}(t)=\frac{1}{m}\sum_{j=1}^m I(\widehat Y^x_{nj} \le t, \widehat \delta^x_{nj}=1).
\end{align}
The Nelson-Aalan estimator \citep{
nelson1972, aalen1978}
 of the conditional cumulative hazards function of $T$ given $X=x$ is
\begin{align}
\label{cH2e}
\widehat\Lambda^x_{nm}(t)=
\int_{[0,t]} \frac{1}{1-\widehat H^x_{nm-}(u)} d\widehat H^x_{1nm}(u).
\end{align}
This can be written more explicitly as
\begin{align*}
\widehat\Lambda^x_{nm}(t)= \sum_{j: \widehat Y^x_{(nj)}\le t}
\frac{\widehat \delta^x_{(nj)}}{m-j+1},
\end{align*}
where $\widehat Y^x_{(n1)} \le \cdots \le \widehat Y^x_{(nm)}$ are the ordered values
of $\widehat Y^x_{n1}, \ldots, \widehat Y^x_{nm}$ and $\widehat \delta^x_{(nj)}$ is the censoring indicator associated with $\widehat Y^x_{(nj)}.$

The Kaplan-Meier product limit estimator \citep{kaplan1958} of the conditional survival function
of $T$ given $X=x$ is
\begin{align}
\label{KM1e}
\widehat{S}_{nm}^x (t)
=\prod_{j: \widehat Y^x_{(nj)}\le t} \left(1-\frac{\widehat \delta^x_{(nj)}}{m-j+1}\right).
\end{align}

\section{Consistency}
%

In this section, we study the convergence properties  of the proposed method.
For simplicity, we write the generator $G=(G_1, G_2)$, and the  generator network
$G_{\vtheta}=(G_{\vtheta_{1}}, G_{\vtheta_{2}})$.
We make the following assumptions.

\begin{assumption}\label{asp2}
	For some $\gamma>0$, $(X, Y)$ satisfies the first moment tail condition
	\begin{align*}
	P(\|(X, Y)\|
>t)=O(1)  t^{-1}\exp({-t^{1+\gamma}/(d+2)}),\quad t > 1,
	\end{align*}
	where $\|\cdot\|$ denotes the Euclidean norm.
\end{assumption}

\begin{assumption}\label{asp3}
	The noise distribution $P_{\eta}$ is absolutely continuous with respect to the Lebesgue measure.
\end{assumption}
\begin{assumption}\label{assume:truncation}
	Only the observations for which $Y_i$, $1\leq i \leq n$ is in the interval $[0, \tau]$ are used.
	There exists a small positive  constant $\delta$ such  that $P(Y> \tau\mid X)>\delta$ holds almost surely with respect to the probability of $X$.
\end{assumption}

Assumption \ref{asp2} is satisfied if $P_{X,Y}$ is subgaussian. Assumption \ref{asp3}
is satisfied by commonly used reference distributions such as normal and uniform distributions.
Assumption \ref{assume:truncation} is commonly assumed in the literature of survival analysis, see, for example, \citet{ag1982}.

To describe the convergence properties, we need the notion of
the integral probability metric (IPM, \citet{muller1997}) between two probability distributions
$\mu$ and $\nu$ with respect to a symmetric evaluation class
$\cF$, defined by
$$
d_{\cF}(\mu, \nu)=\sup_{f \in \cF}\{E_{\mu} f-E_{\nu} f\}.
$$
By specifying the evaluation function class $\mathcal{F}$, we
obtain different metrics for the space of probability distributions. For example,
\begin{itemize}
	\item $\mathcal{F}=$   bounded Lipschitz function class $: d_{\mathcal{F}}=d_{BL}$,
	this yields the bounded Lipschitz
	metric \citep{dudley2018real}. This metric
	characterizes weak convergence, that is, a sequence of probability measures $\{\mu_n\}$ converges
	to $\mu$ if and only if $d_{\cF}(\mu_n, \mu) \to 0$.
	\item $\mathcal{F}= $ Lipschitz function class. This results in the Wasserstein distance $d_{\mathcal{F}}=d_{W_1}$, which is used in the training of Wasserstein GAN \citep{arjovsky17}.
\end{itemize}

Denote the joint distribution of
$(X, G(\eta,X))\equiv (X, G_1(\eta, X), G_2(\eta, X))$ by $P_{X, G}.$
We consider the convergence of the conditional generator with respect to
the bounded Lipschitz metric
\[
d_{\cF^1_B}(P_{X, G}, P_{X,Y, \Delta}) = \sup_{f \in \cF^1_B} \{\Ebb_{(X, \eta)\sim P_X P_{\eta}}
f(X, G(\eta, X)) - \Ebb_{(X, Y, \Delta)\sim P_{X,Y, \Delta}} f(X, Y, \Delta)\},
\]
where $\cF^1_B$ is the uniformly bounded 1-Lipschitz function class,
\begin{align*}
\mathcal{F}^1_B=\Big\{f:  & \mathbb{R}^{d+1} \times [0,1] \mapsto \mathbb{R},
|f(z_1)-f(z_2)|\leq \|z_1-z_2\| , z_1, z_2 \in  \mathbb{R}^{d+1}\times [0,1],
 \|f\|_{\infty}\leq B\Big \}
\end{align*}
for some constant $0 < B < \infty$. If $P_{X,Y,\Delta}$ has a bounded support, then $d_{\text{BL}}$ is essentially the same as the 1-Wasserstein distance.

Both the generator and the discriminator are assumed to be implemented by the  feedforward  ReLU neural networks, which can be represented in the form of
$g=g_{L}\circ g_{L-1}\circ\cdots\circ g_{0},$
where $g_{i}(x)=\sigma(V_ix+b_i)$  and  $g_{L}(x)=V_{L}x+b_{L}$, with  the weight matrices $V_{i}\in\real^{d_{i+1}\times d_{i}}$, the bias vectors $b_{i}\in\real^{d_{i+1}\times 1}$  for $i=0,\dots, L$, and $\sigma$  the ReLU activation function.
So the network is parameterized by $(V_0,\dots, V_{L}, b_{0},\dots, b_{L})$ with  depth (number of hidden layers) $L$ and  width   (the maximum width of all hidden layers)  $W=\max\{d_1,\dots, d_{L}\}$.  Let $(L_1, W_1)$ be the depth and width of
the discriminator network $D_{\vphi}$ and let $(L_2,W_2)$ be the depth and width of
the generator network $G_{\vtheta}$.

Denote the parameter of the conditional generator   by $\vtheta_1=(V_{1,0},\dots, V_{1,L_2},b_{1,0},\dots, b_{1,L_2})$ and $\vtheta_2=(V_{2,0},\dots, V_{2,L_2},b_{2,0},\dots, b_{2,L_2})$.
For  a matrix $V=(v_{i,j}),$
 let $\|V\|_{2}=\sup_{\|x\|=1}\|Vx\|$ be the 2-norm of $V$.
 {\color{black} Recall that the conditional generator $\hG_n$ is defined in (\ref{3}).}

\begin{theorem}	\label{lemm2a}
Let  $(L_1, W_1)$ of $D_{\vphi}$ and $(L_2, W_2)$ of $G_{\vtheta}$ be specified such that
	$W_1 L_1=\ceil*{\sqrt{n}}$ and $W_2^2 L_2= c q n $ for a constant $12\leq c\leq 384$.
	Suppose that  Assumptions \ref{asp2} and \ref{asp3} hold and
$\|G_{\vtheta}\|_{\infty} \le 1+\log n.$
Also, suppose that the weight matrices satisfy
$\prod_{j=0}^{L_2}\|V_{i,j}\|_2\leq K_1$, $i=1,2$, for some positive constant $K_1$. Then, we have
	\begin{equation}
	\label{ProConverge1}
	d_{\cH^1_B}(P_{\tG_n(\eta,X)}, P_{(Y,\Delta)|X})=\sup_{h\in \cH^1_B}\Ebb_{\eta}(h(\tG_n(\eta, X)))-\Ebb(h(Y,\Delta)|X)\to_P 0,
	\end{equation}
where $P_{\tG_n}$ represents the probability with respect to the randomness in $\tG_n.$
	Let $r_n= (d+2)^{{1}/{2}} n^{-{1}/{(d+2)}}\log n$, we have
	\begin{equation}\label{Converge2}
	\Ebb_{\tG_n}\big\{\Ebb_X(d_{\cH^1_B}(P_{\tG_n(\eta,X)}, P_{(Y,\Delta)|X}) )\big\}
	\lesssim    \{\log r_n^{-1}\}^{-1/2}\{\log\{\log r_n^{-1}\}\}^{1/2},
	\end{equation}
where $\Ebb_{\tG_n}$ means the expectation with respect to the randomness in $\tG_n.$
\end{theorem}

In the theorem,  we restrict the value of $\|G_{\vtheta}\|$ to  $[0, 1+\log n].$ This
requirement can be automatically satisfied by using a clipping layer $\ell$ as the output layer of the network \citep{lzjh2021},
$
\ell(a)=a\wedge c_{n} \vee (-c_{n})=\sigma(a+c_{n})-\sigma(a-c_{n})-c_{n},
$
where $c_{n}=1+\log n.$ In addition, we restrict the parameters of $G_{\vtheta}$ to satisfy $\prod_{j=0}^{L_2}\|V_{i,j}\|_2\leq K_1$ for $i=1,2$, from which we guarantee   $\|G_{\vtheta}\|_{\operatorname{Lip}}\leq \sqrt{2}K_1$.
The condition can  be  satisfied when only a small number of  parameters have large value, and similar conditions are assumed in the literature on WGAN (e.g., \citep{arjovsky17,
Biau2021, BiauWGAN2021}. In practice, different ways have been explored to restrict the Lipschitz constant of the network, including weight clipping  \citep{arjovsky17} and
gradient regularization \citep{Gulrajani2017}.


The following theorem establishes the consistency of $\widehat \Lambda^x_n$
and $\widehat S^x_n$ in an appropriate sense.
\begin{theorem}
	\label{thm2}
	Suppose Assumptions \ref{asp1} -- \ref{assume:truncation} are satisfied and the  conditions of Theorem  \ref{lemm2a} hold. Then, we have for every $\varepsilon>0$,
	\begin{align}
	\label{HazardConverge}
&	\Ebb_{X}(P_{\tG_n}(	\sup_{0\le t \le \tau} |\widehat \Lambda^X_n(t) -\Lambda^X (t)|>\varepsilon))\rightarrow 0, \\
	\label{SurvivalConverge}
&	\Ebb_{X}(P_{\tG_n}(		\sup_{0 \le t < \tau} |\widehat S^X_n(t)-S^X(t)| >\varepsilon))\rightarrow 0
	\end{align}
	as $n \to \infty$.
\end{theorem}
Theorem \ref{thm2} shows that the estimated conditional hazard function $\widehat \Lambda^x_n$
and the estimated conditional survival function $\widehat S^x_n$ are consistent in expectation with
to the distribution of $X$.
%


We now turn to the convergence properties of the estimators $\hat{\Lambda}_{nm}^x$ and $\hat{S}_{nm}^x$, which are computed by sampling from the conditional generator $\hG_n.$
We first establish the following lemma that gives the convergence rate of $\hat{\Lambda}_{nm}^x$ and $\hat{S}_{nm}^x$ in terms of the size of the generated conditional samples.
\begin{lemma}
	\label{lem:KM}	For given conditional generator $\hat{G}_n$ and $x\in\cX$, let $\tau_n^x $ be the point such   that $ 1-\hat{H}_n^x(\tau_n^x) >\nu$   and $\hat{H}_n^x(\tau_n^x) >1/2$ for a small positive constant $\nu$. Then we have
\begin{align*}
&P_{\eta}\big(\sup_{t\in[0,\tau^x_n]}|\hat{\Lambda}_{nm}^x(t)-\hat{\Lambda}_n^x(t)|=O(\sqrt{ \log m/ m })  \big)=1, \text{ and } \\
&P_{\eta}\big(\sup_{t\in[0,\tau^x_n]}|\hat{S}_{nm}^x(t)-\hat{S}_n^x(t)|=O(\sqrt{ \log m/ m })   \big)=1,
\end{align*}
	where $P_{\eta}$ represents the probability with respect to the randomness in $\eta$.
\end{lemma}
Lemma \ref{lem:KM} holds for any positive constant $c$ such that   $\hat{H}_n^x(\tau_n^x) >c.$ We take $c=1/2$ in  Lemma \ref{lem:KM} for the convenience of proof.
From the proof of Theorem \ref{thm2}, we have    $1-\hat{H}_n^x(\tau)>\delta/2$ holds  with probability tending to 1 as  $n\rightarrow \infty$. Hence,  Lemma \ref{lem:KM} leads to the result below.
\begin{theorem}
	\label{thm_together}
	Under the conditions of Theorem \ref{thm2}, we have
\[
\sup_{0\le t \le \tau} |\hat{\Lambda}_{nm}^X(t) -\Lambda^X (t)| \to_P  0
\ \text{ and } \
\sup_{0 \le t \le \tau} |\hat{S}_{nm}^X(t)-S^X(t)| \to_P   0,
\]
as $n,m \to \infty$.
\end{theorem}
Therefore, the estimators of the conditional cumulative hazard function and
the conditional survival function computed based on Monte Carlo are consistent.

\section{Implementation}
\label{implementation}
Our method is implemented in two steps. The first step is training the conditional generator, i.e. getting the estimator of the neural network parameters. The second step is generating samples from the conditional generator at given $X = x$ and applying the Nelson-Aalen and Keplan-Meier estimator to the samples.

Let $(G_{\vtheta_1}, G_{\vtheta_2})$ and $D_\vphi$ denote the conditional generator for $(Y,\Delta)$ and the discriminator. Let $\vtheta_1, \vtheta_2$ and $\vphi$ denote their parameters in neural networks, respectively.
The estimator $(\hat\vtheta_1, \hat\vtheta_2,\hat\vphi)$ of the neural network parameter $(\vtheta_1, \vtheta_2, \vphi)$
is the solution to the minimax problem
\begin{align}
	\label{MinimaxA}
	(\vtheta_1, \vtheta_2, \vphi) = \argmin_{\vtheta_1,\vtheta_2}\argmax_{\vphi}
\frac{1}{n}\sum_{i=1}^n\{
	& D_\vphi(X_i,\hG_{\vtheta_1}(\eta_i, X_i), \hG_{\vtheta_2}(\eta_i, X_i))-
	D_\vphi(X_i,Y_i, \Delta_i) \nonumber \\
	& -\lambda (\|\nabla_{(x, y)}D_\vphi(X_i,Y_i, \Delta_i)\|_2-1)^2\big\},
\end{align}
where  $\nabla_{(x,y)}D_\vphi(X_i,Y_i, \Delta_i)$ is the gradient of
$D_\vphi(x, y, \delta)$ with respect to $(x, y)$ evaluated at $(X_i, Y_i, \Delta_i).$
Here we use the gradient penalty algorithm to impose the constraint that the discriminator belongs to the class of 1-Lipschitz functions \citep{gulrajani2017improved}.
The minimax problem (\ref{MinimaxA}) is solved by
updating $(\vtheta_1,\vtheta_2)$ and $\vphi$ alternately as follows:

(a) Fix  $(\vtheta_1, \vtheta_2)$,  update the discriminator by maximizing the empirical objective function
\begin{align*}
	\hat\vphi = \argmax_{\vphi} \frac{1}{n}\sum_{i=1}^n\{
	& D_\vphi(X_i,\hG_{\vtheta_1}(\eta_i, X_i), \hG_{\vtheta_2}(\eta_i, X_i))-
	D_\vphi(X_i,Y_i, \Delta_i) \\
	& -\lambda (\|\nabla_{(x, y)}D_\vphi(X_i,Y_i, \Delta_i)\|_2-1)^2\big\}
\end{align*}
with respect to $\vphi.$
The third term on the right side is the gradients penalty for the Lipschitz condition on the discriminator  \citep{gulrajani2017improved}.


(b) Fix  $\vphi$,  update the generator by minimizing the empirical objective function
\begin{eqnarray*}
	(\hat\vtheta_1, \hat\vtheta_2)= \argmin_{\vtheta_1,\vtheta_2}\frac{1}{n}\sum_{i=1}^n
	\hD_\vphi(X_i, G_{\vtheta_1}(\eta_i, X_i), G_{\vtheta_2}(\eta_i, X_i))
\end{eqnarray*} with respect to $\vtheta_1$ and $\vtheta_2$.

The conditional generator $G_1$ for the observed time $Y$ has a linear output activation function.  The conditional generator $G_2$ for the censoring indicator has a sigmoid output activation function, then the censoring status is determined based on whether the value of the sigmoid function is greater than 0.5 or not. The two generators $G_1$ and $G_2$ can have the same structure and share the same parameters except for the output layer. We implemented this algorithm in TensorFlow \citep{abadi2016tensorflow}.



\section{Numerical studies}
In this section, we conducted simulation studies to evaluate the finite-sample performance of the proposed method and illustrate its applications to two datasets.

\subsection{Simulation studies}
We compare the results based on GCSE and the Cox proportional hazards (PH) model under covariate-dependent and -independent censoring scenarios for four generating models described below. For all eight simulation scenarios, we considered covariate $X \sim N(\0, \bI_5)$ and censoring rate at 50\%.

\begin{itemize}
	\item Model (M1): A proportional hazards model with a constant baseline hazard function. The conditional hazard of survival time $T$ given $X=x$ is
$
\lambda(t|x) = \lambda_0 \exp (\beta(x)),
$
where $\lambda_0=1$ and $\beta(x) = x_1 + 0.5x_2+1.5x_3 - 2x_4 - 0.3 x_5.$
For independent censoring, the censoring times $C \sim \mbox{Uniform}(0, 2)$. For covariate-dependent censoring, the censoring times follow an exponential distribution with rate parameter $0.8 e^{\phi(x)}$ truncated at 10, where $\phi(x)=x_1+x_2-x_3$.

	
\item Model (M2): A nonlinear proportional hazards model.
The conditional distribution of survival time $T$ given $X=x$ follows a Weibull distribution with conditional survival function $S(t|x) = \exp [- ({t}/{\psi(x)})^v], t \ge 0$, where the shape parameter $\nu=2$ and the scale parameter $\psi(x) = \exp (x_1^2 - 0.5x_2^2 + 1.5x_3 - 2x_4 - 0.3 x_5)$.
The covariate-independent censoring times $C \sim \mbox{Uniform} (0,3.5)$ and covariate-dependent censoring times follow an exponential distribution with rate parameter $0.4 e^{\phi(x)}$ truncated at 10, where $\phi(x)=x_1+x_2-x_3$.

\item Model (M3): An accelerated failure time model with a normal error: $\log T = x_1^2 +0.5 x_2-0.8 x_3 + \epsilon$, where $\epsilon \sim N(0, 1)$.
For censoring independent of covariates, we set censoring times $C \sim \mbox{Uniform}(0, 6)$. Under the covariate-dependent scenario, the censoring times follow an exponential distribution with rate parameter $0.2 e^{\phi(x)}$ truncated at 10, where $\phi(x)=x_1+x_2-x_3$.

\item Model (M4): An accelerated failure time model with an log(exponential) error component: $\log T=x_1^2 + 0.5x_2 - 0.8x_3 + \epsilon$, where $\exp(\epsilon) \sim \mbox{Exponential}(1)$. The covariate-independent censoring times $C \sim \mbox{Uniform} (0,5)$, whereas the covariate-dependent censoring times follow an exponential distribution with rate parameter $0.3 e^{\phi(x)}$ truncated at 10, where $\phi(x)=x_1+x_2-x_3$.
	
\end{itemize}


The observed data $D=(X, Y, \Delta)$, where $Y$ is the natural logarithm of $\min \{T, C\}$ and $\Delta=1\{ T \le C\}$. For each model, we set sample size $n=10,000$ and repeated the simulation 200 times.
We used feedforward neural networks with two hidden layers for both the generator and the discriminator. For covariate-independent censoring scenarios, we used $\{60, 30\}$ nodes for both the discriminator and generator for models (M1) and (M2). For (M3) and (M4), we used $\{50, 25\}$ nodes for the discriminator and $\{40, 20\}$ nodes for the generator. Under covariate-dependent censoring scenarios, we used $\{50,25\}$ for both discriminator and generator for (M1). Nodes chosen for (M2), (M3), and (M4) are the same as those under covariate-independent scenarios. We adopted the Elu activation function for hidden layers with $\alpha = 0.3$, which is defined as $\sigma_{\alpha}(x)=x$ if $x \ge 0$ and $\sigma(x)=\alpha (e^x -1)$ if $x < 0$ with $\alpha >0$.
Dimensions of the noise vector $\eta$ for models (M1)--(M4) are 3, 7, 5, and 7, respectively. For comparison, we used the \textsf{R} package \textsl{survival} to obtain estimation results based on the PH model.

Tables \ref{tab:sprob1} and \ref{tab:sprob2} compare the estimation results obtained by the GCSE method and the partial likelihood approach with the PH model  based on 200 replications under covariate-independent and -dependent censoring scenarios at theoretical levels 25\%, 50\% and 75\% for $x=\{({\bf -0.5}_5)^\top, ({\bf 0.5}_5)^\top, ({\bf 1.0}_5)^\top\}$, respectively. Intuitively, estimates close to the theoretical quantiles imply good performance. Similar results were observed under covariate-independent and -dependent censoring settings for both the GCSE method and the PH model. Although the PH model predictably outperformed the GCSE method in (M1), the obtained estimates by the GCSE  were still reasonably close to the true quantiles. Moreover, in (M2)--(M4), the GCSE method yielded estimates close proximity to the true quantiles while the partial likelihood approach  produced estimates far from the true quantiles, due to violation of the proportional hazards assumption.


Figures \ref{fig:simu_survival1} and \ref{fig:simu_survival2} show one instance of the estimated conditional survival functions under the covariate-independent and -dependent censoring scenario for $x=\{({\bf -0.5}_5)^\top, ({\bf 0.5}_5)^\top, ({\bf 1.0}_5)^\top\}$, respectively. The solid grey line represents the ground-truth, while the dashed line and dotted line respectively depict the survival functions estimated by the GCSE method and PH model. Evidently, survival functions estimated by the GCSE method are within close proximity to the true function for simulation scenarios. On the other hand, the estimated survival functions from the PH model are close to the true function in (M1), however, its performance deteriorates in (M2)--(M4) due to violation of the proportionality assumption. Similar results are observed for both censoring settings. The GCSE method consistently yields promising results regardless of the underlying model while the PH model is obviously constrained by the proportional hazards assumption.


\begin{table}[H]
\caption{Estimated conditional survival probabilities for covariate-independent censoring scenarios based on 200 simulations for $x=\{({\bf -0.5}_5)^\top, ({\bf 0.5}_5)^\top, ({\bf 1.0}_5)^\top\}$ at theoretical percentiles 25\%,  50\% and 75\%. The closer to theoretical levels the better. GCSE denotes the proposed generative conditional survival function estimator and PH represents the Cox proportional hazards model.}

\bigskip
	\centering \begin{tabular}{r@{\hspace{2mm}}c@{\hspace{2mm}}c@{\hspace{2mm}}c@{\hspace{2mm}}
c@{\hspace{2mm}}c@{\hspace{2mm}}c@{\hspace{2mm}}c}
		\hline
		&\multicolumn{3}{c}{GCSE} & \multicolumn{3}{c}{PH}\\
$x$	& 25\% &  50\%  &  75\%  & &  25\%  &  50\%  & 75\% \\
\hline
		& \multicolumn{7}{c}{(M1) Cox-Exponential}\\
		\hline
-0.5 & 25.66 (0.46) & 51.08 (0.38) & 75.96 (0.27) & & \textbf{24.52} (0.16) & \textbf{49.58} (0.16) & \textbf{75.44} (0.09) \\
  0.5 & 18.36 (0.81) & \textbf{49.83} (0.40) & 73.79 (0.28) & & \textbf{24.84} (0.30) & 49.10 (0.15) & \textbf{74.06} (0.11) \\
  1 & \textbf{25.04} (1.34) & 47.03 (0.90) & 73.84 (0.50) & & 37.25 (0.38) & \textbf{49.89} (0.22) & \textbf{74.49} (0.14) \\
		\hline
		&\multicolumn{7}{c}{(M2) Cox-Weibull}\\
		\hline
-0.5 & \textbf{27.15}(0.60) & \textbf{52.24}(0.53) & \textbf{77.85}(0.43) & & 51.35(0.22) & 58.44(0.20) & 66.83(0.17) \\
  0.5 & \textbf{30.30}(0.68) & \textbf{55.68}(0.63) & \textbf{79.42}(0.43) & & 50.24(0.18) & 58.15(0.17) & 66.76(0.15) \\
  1 & \textbf{32.71}(1.12) & 55.33(1.12) & \textbf{78.46}(0.69) & & 40.69(0.33) & \textbf{49.2}(0.32) & 59.48(0.29) \\
		\hline
		&\multicolumn{7}{c}{(M3) AFT-Normal}\\
		\hline
-0.5 & \textbf{26.60}(0.39) & \textbf{51.28}(0.38) & \textbf{75.13}(0.36) & & 48.01(0.14) & 64.94(0.12) & 79.58(0.08) \\
  0.5 & \textbf{26.09}(0.32) & \textbf{49.52}(0.34) & \textbf{74.07}(0.33) & & 47.78(0.14) & 66.03(0.11) & 81.07(0.08) \\
  1 & \textbf{26.67}(0.70) & \textbf{54.28}(0.61) & \textbf{77.54}(0.47) & & 27.92(0.22) & 45.69(0.21) & 64.68(0.17) \\
\hline
		&\multicolumn{7}{c}{(M4) AFT-Exponential}\\
		\hline
-0.5 & \textbf{27.34}(0.39) & \textbf{51.92}(0.36) & \textbf{75.48}(0.28) & & 44.76(0.13) & 61.82(0.11) & 79.15(0.08) \\
  0.5 & \textbf{25.98}(0.37) & \textbf{49.93}(0.38) & \textbf{74.71}(0.30) & & 45.16(0.12) & 62.78(0.10) & 80.29(0.07) \\
  1 & \textbf{26.42}(0.74) & \textbf{53.71}(0.65) & \textbf{77.46}(0.45) & & 26.53(0.20) & 43.87(0.18) & 66.32(0.14) \\
		\hline
\end{tabular}

\label{tab:sprob1}
\end{table}

\begin{table}[H]
\caption{Estimated conditional survival probabilities for covariate-dependent scenarios based on 200 simulations for $x=\{({\bf -0.5}_5)^\top, ({\bf 0.5}_5)^\top, ({\bf 1.0}_5)^\top\}$ at theoretical percentiles 25\%,  50\% and 75\%. The closer to theoretical levels the better. GCSE denotes the proposed generative conditional survival function estimator and PH represents the Cox proportional hazards model.}

\bigskip
	\centering \begin{tabular} {r@{\hspace{2mm}}c@{\hspace{2mm}}c@{\hspace{2mm}}c@{\hspace{2mm}}
c@{\hspace{2mm}}c@{\hspace{2mm}}c@{\hspace{2mm}}c}
		\hline
		&\multicolumn{3}{c}{GCSE} & \multicolumn{3}{c}{PH}\\
$x$	& 25\% &  50\%  &  75\%  & &  25\%  &  50\%  & 75\% \\
		\hline
		& \multicolumn{7}{c}{(M1) Cox-Exponential}\\
		\hline
-0.5 & \textbf{25.69} (0.43) & \textbf{50.48} (0.40) & \textbf{74.99} (0.30) & & 23.64 (0.18) & 49.39 (0.16) & 74.95 (0.11) \\
  0.5 & 19.90 (0.79) & 46.60 (0.52) & 72.22 (0.33) & & \textbf{24.28} (0.21) & \textbf{49.10} (0.20) & \textbf{74.50} (0.12) \\
  1 & \textbf{25.24} (1.57) & 41.88 (1.43) & 69.45 (0.75) & & 24.74 (0.27) & \textbf{49.53} (0.26) & \textbf{73.98} (0.16) \\
		\hline
		&\multicolumn{7}{c}{(M2) Cox-Weibull}\\
		\hline
-0.5 & \textbf{33.93} (0.71) & 58.45 (0.61) & \textbf{81.45} (0.39) & & 49.82 (0.19) & \textbf{58.28} (0.18) & 68.05 (0.14) \\
  0.5 & \textbf{28.07} (0.67) & \textbf{52.62} (0.59) & \textbf{76.76} (0.41) & & 44.11 (0.20) & 53.76 (0.18) & 64.34 (0.17) \\
  1 & \textbf{22.11} (0.91) & \textbf{42.83} (1.01) & \textbf{69.34} (0.77) & & 30.89 (0.32) & 40.73 (0.32) & 53.04 (0.31) \\
		\hline
		&\multicolumn{7}{c}{(M3) AFT-Normal}\\
		\hline
-0.5 & \textbf{26.48} (0.35) & \textbf{51.42} (0.39) & \textbf{75.40} (0.34) & & 48.27 (0.15) & 65.25 (0.12) & 79.71 (0.08) \\
  0.5 & \textbf{26.17} (0.32) & \textbf{50.03} (0.36) & \textbf{74.80} (0.34) & & 47.63 (0.14) & 65.96 (0.11) & 81.01 (0.08) \\
  1 & \textbf{25.67} (0.65) & \textbf{53.48} (0.66) & \textbf{77.21} (0.51) & & 27.47 (0.22) & 45.28 (0.22) & 64.42 (0.18) \\
\hline
		&\multicolumn{7}{c}{(M4) AFT-Exponential}\\
		\hline
-0.5 & \textbf{24.99} (0.43) & \textbf{49.31} (0.41) & \textbf{73.67} (0.29) & & 44.66 (0.14) & 61.65 (0.12) & 79.43 (0.08) \\
  0.5 & \textbf{25.96} (0.45) & \textbf{50.80} (0.44) & \textbf{75.27} (0.30) & & 39.72 (0.14) & 58.56 (0.12) & 78.11 (0.09) \\
  1 & \textbf{23.78} (0.91) & \textbf{49.38} (0.86) & \textbf{75.30} (0.54) & & 19.93 (0.21) & 35.92 (0.21) & 60.53 (0.18) \\
		\hline
\end{tabular}

\label{tab:sprob2}
\end{table}

\begin{figure}[H]
\begin{center}
\begin{minipage}[t]{5in}
\centering
\centerline{\footnotesize{(M1) Cox-Exponential}}
\includegraphics[width=5in,height=1.5in]{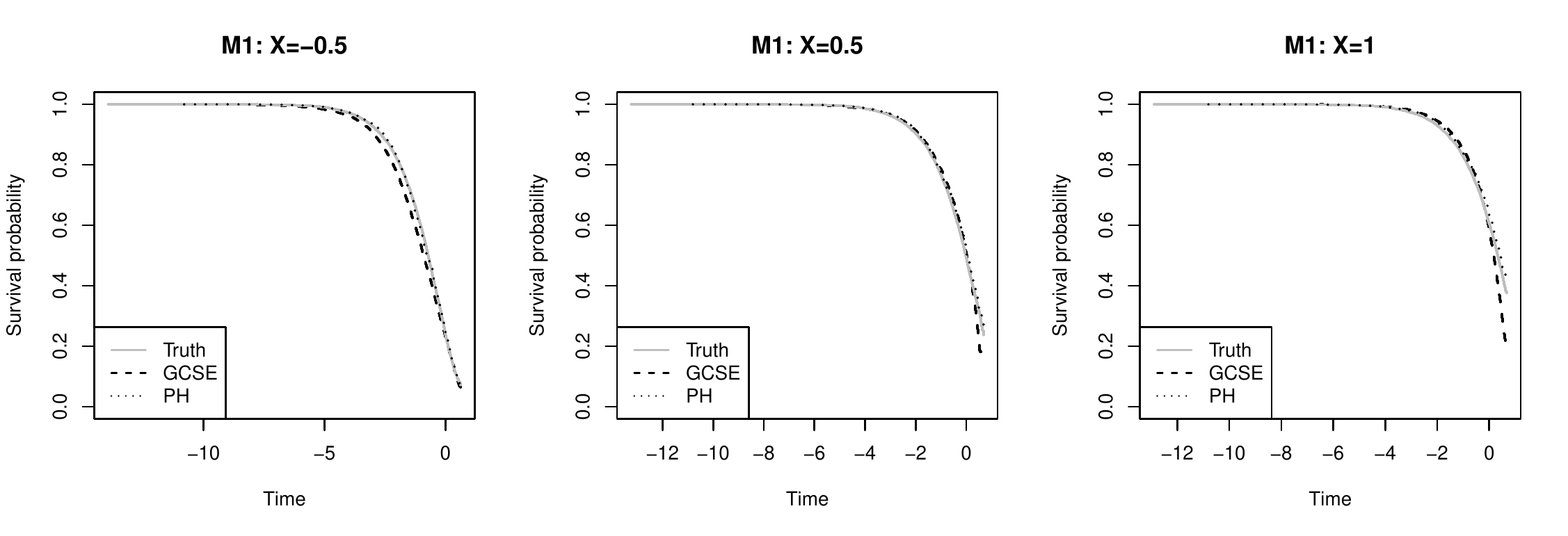}
\end{minipage}

\begin{minipage}[t]{5in}
\centering
\centerline{\footnotesize{(M2) Cox-Weibull}}
\includegraphics[width=5in,height=1.5in]{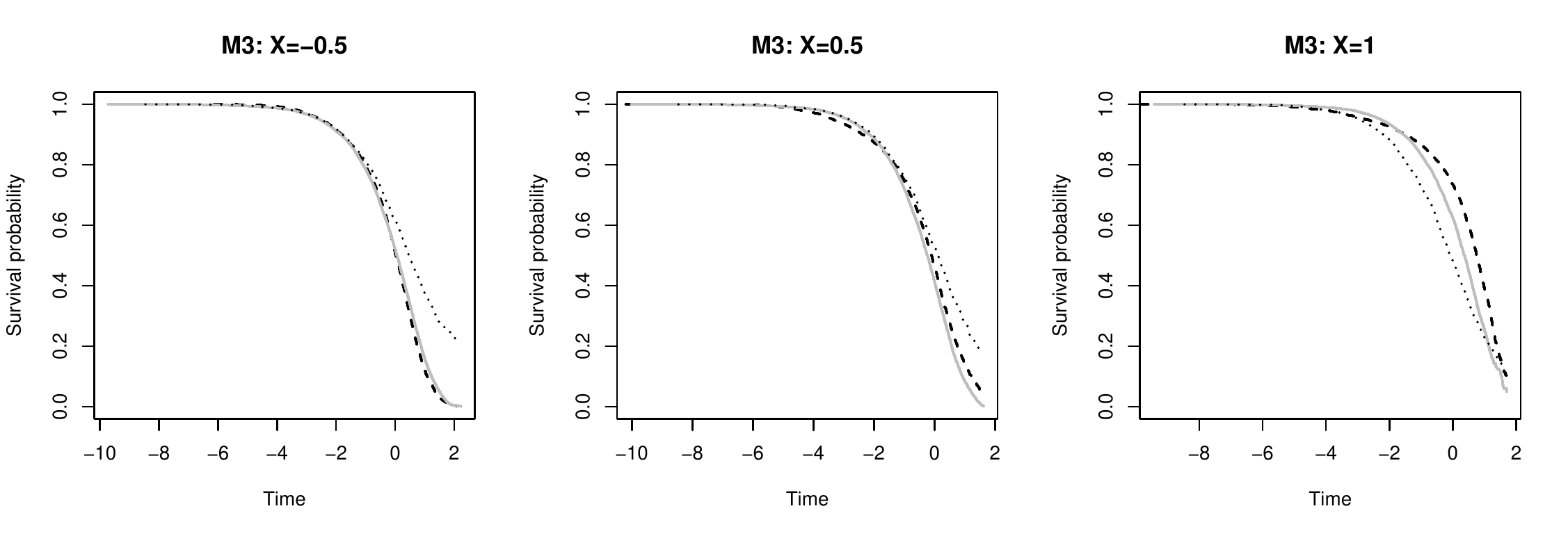}
\end{minipage}

\begin{minipage}[t]{5in}
\centering
\centerline{\footnotesize{(M3) AFT-Normal}}
\includegraphics[width=5in,height=1.5in]{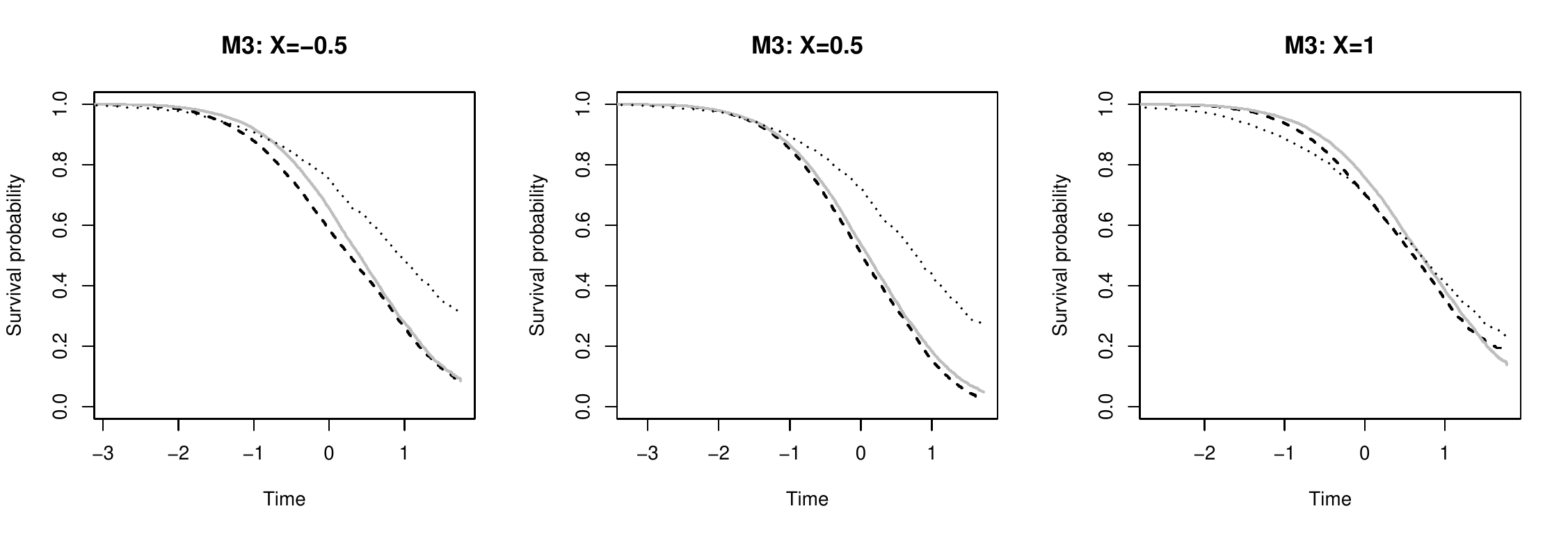}
\end{minipage}

\begin{minipage}[t]{5in}
\centering
\centerline{\footnotesize{(M4) AFT-Exponential}}
\includegraphics[width=5in,height=1.5in]{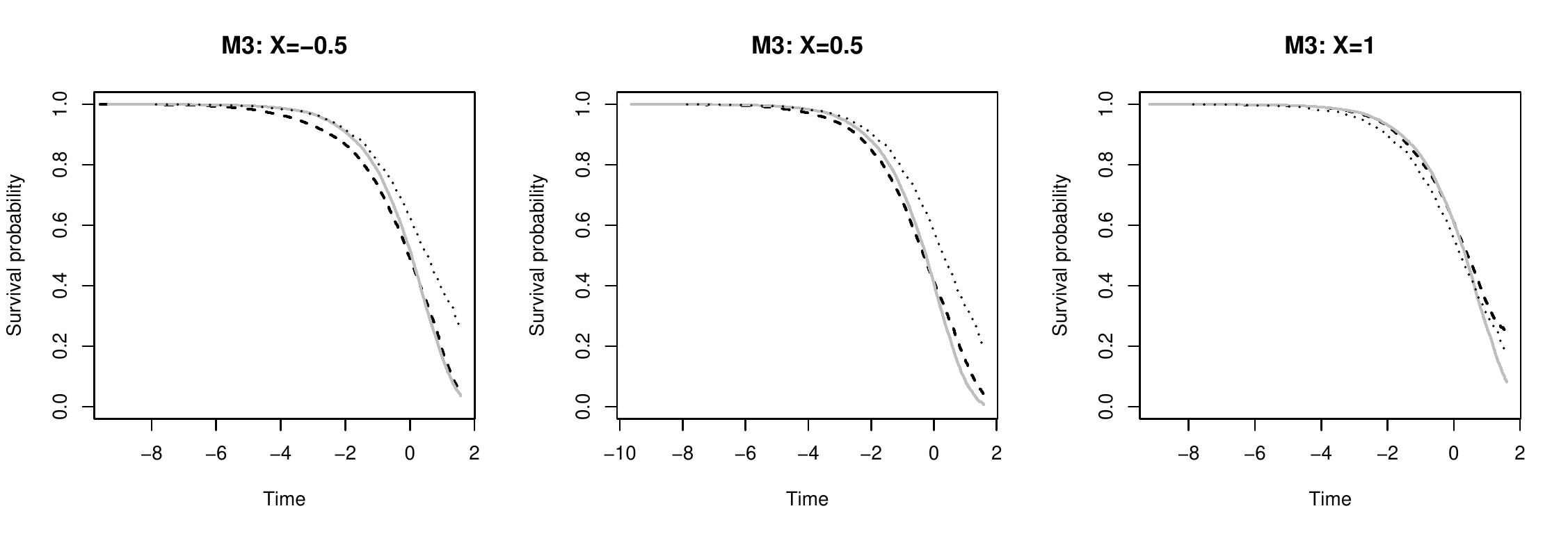}
\end{minipage}

\caption{The true survival function (solid gray line) and estimated conditional survival functions for covariate-independent scenarios based on the proposed method (dashed line) and the Cox regression (dotted line) for $x=\{({\bf -0.5}_5)^\top, ({\bf 0.5}_5)^\top, ({\bf 1.0}_5)^\top\}$.}
\label{fig:simu_survival1}
\end{center}
\end{figure}

\begin{figure}[H]
\begin{center}
\begin{minipage}[t]{5in}
\centering
\centerline{\footnotesize{(M1) Cox-Exponential}}
\includegraphics[width=5in,height=1.5in]{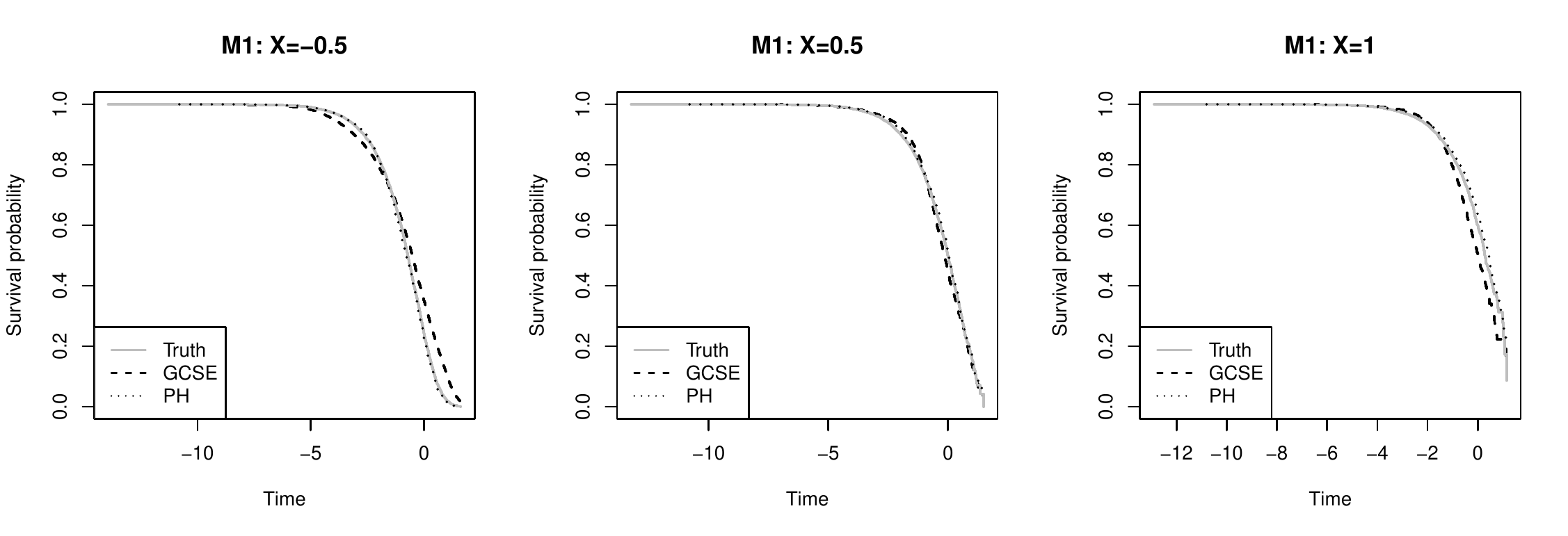}
\end{minipage}

\begin{minipage}[t]{5in}
\centering
\centerline{\footnotesize{(M2) Cox-Weibull}}
\includegraphics[width=5in,height=1.5in]{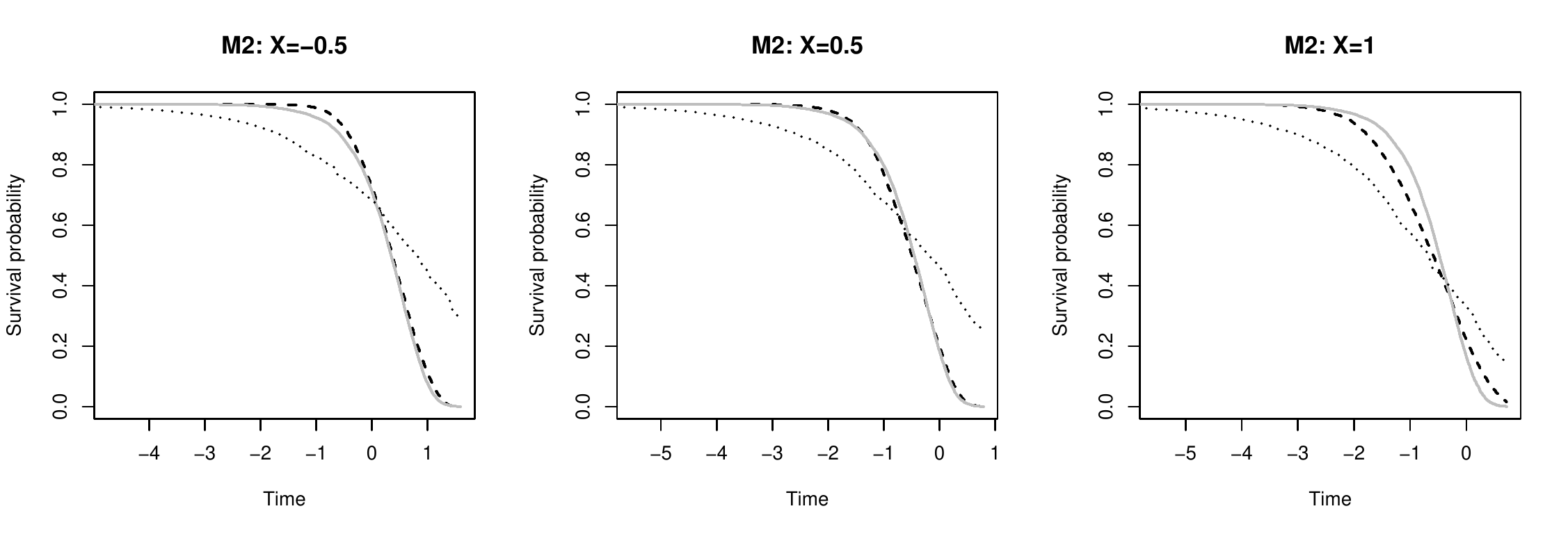}
\end{minipage}

\begin{minipage}[t]{5in}
\centering
\centerline{\footnotesize{(M3) AFT-Normal}}
\includegraphics[width=5in,height=1.5in]{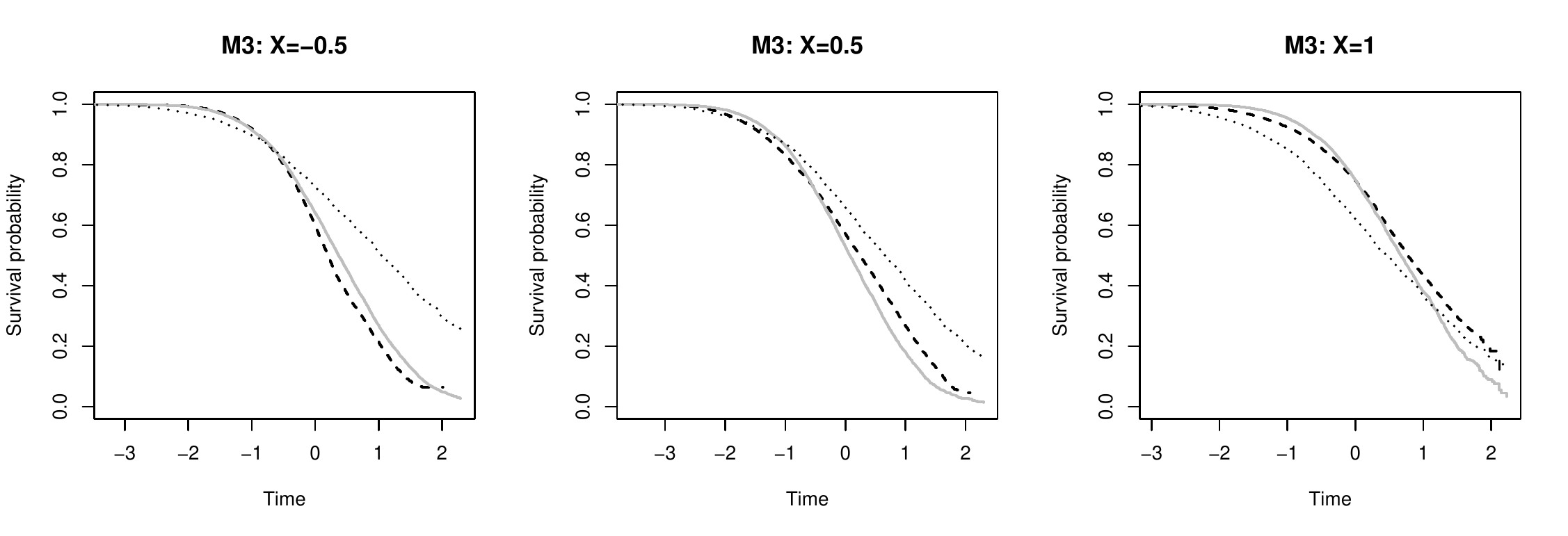}
\end{minipage}

\begin{minipage}[t]{5in}
\centering
\centerline{\footnotesize{(M4) AFT-Exponential}}
\includegraphics[width=5in,height=1.5in]{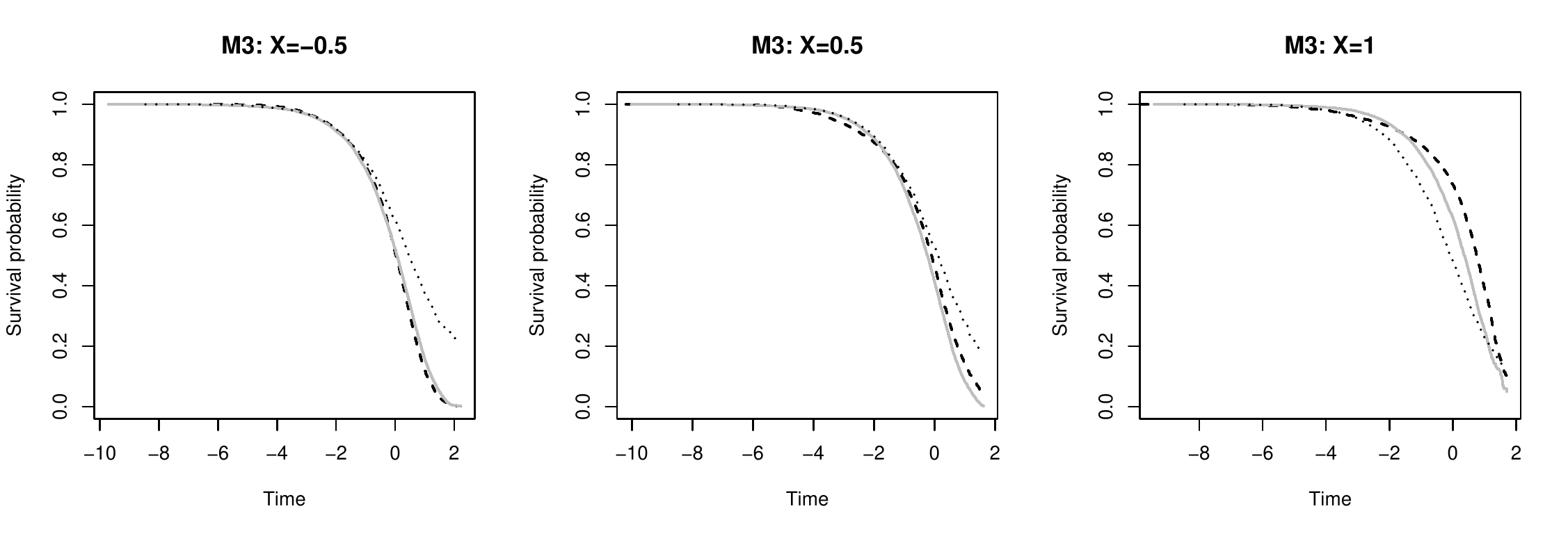}
\end{minipage}

\caption{The true survival function (solid gray line) and estimated conditional survival functions for covariate-dependent scenarios based on the proposed method (dashed line) and the Cox regression (dotted line) for $x=\{({\bf -0.5}_5)^\top, ({\bf 0.5}_5)^\top, ({\bf 1.0}_5)^\top\}$.}
\label{fig:simu_survival2}
\end{center}
\end{figure}

\subsection{Data examples} 
We illustrate the application of GCSE with two publicly available censored survival datasets. A training set was used to learn a conditional generator and estimate the conditional survival distribution function, followed by construction of predication intervals for the survival times of patients in the test set to evaluate performance of the estimators.

\subsubsection{The PBC dataset}
Primary biliary cholangitis (PBC) is an autoimmune disease that causes the destruction of small bile ducts in the liver, eventually leading to cirrhosis and liver decompensation. The Mayo Clinic trial in PBC enrolled 424 patients between 1974 and 1984 \citep{fh1991, tg2000} and randomly assigned each subject to placebo or D-penicillamine treatment.
The dataset is conveniently accessible from the \textsf{R} package \textsf{survival}, which contains 17 medical and treatment covariates including \textit{age, albumin, alk.phos, ascites, ast, bili, chol, copper, edema, hepato, platelet, protime, sex, spiders, stage, trt, trig}. After removing erroneous records and missing data, 276 patients remained for analysis with 90\% of data allocated to training and 10\% to testing. Status at the endpoint includes censored, transplant and death. The 25 patients who received transplant were included in the censored group, resulting a total of 165 censored patients, or censoring rate 60\%. The observed \textit{time} was the number of days elapsed since registration until death, transplantation, or end of study period July, 1986, whichever happened first. The covariates were standardized to have zero mean and unit variance while time was log-transformed.

We applied the GCSE method to estimate the conditional survival functions of PBC patients using this dataset. We parameterized the generator and discriminator using feedforward neural networks with two hidden layers and adopted the Elu activation function with $\alpha = 0.3$. The generator and discriminator were both implemented with $\{30,15\}$ nodes.
Figure \ref{fig:pbc_ci} and \ref{fig:pbc_ci_cox} show the GCSE and PH prediction intervals for patients in the test set, respectively. During the testing process, a patient with predicted probability of censoring higher than 50\% was categorized as censored.
Otherwise, a 90\% confidence interval was constructed based on the Kaplan-Meier estimator for patients predicted to be uncensored. The test set contained 28 subjects. Of the 17 censored subjects, 8 were correctly predicted as censored. All of the uncensored subjects were predicted as uncensored and the prediction intervals contained the true survival times for 9 out of 11 subjects. Similarly, PH prediction intervals included 9 out of 11 uncensored subjects, however incapable to correctly predict censoring status for censored subjects.
Figure \ref{fig:pbc_3de_case} shows the estimated survival functions for three patients from the test set. The estimated survival probabilities of these patients at the 25\textsuperscript{th}, 50\textsuperscript{th}, and 75\textsuperscript{th} quantiles of the uncensored survival times in the test set. Evidently, the estimated survival functions using GCSE and PH are similar and yielded reasonable estimates.

 \begin{figure}[H]
 	\centering
 	\includegraphics[width=5.4 in, height=2.2 in]{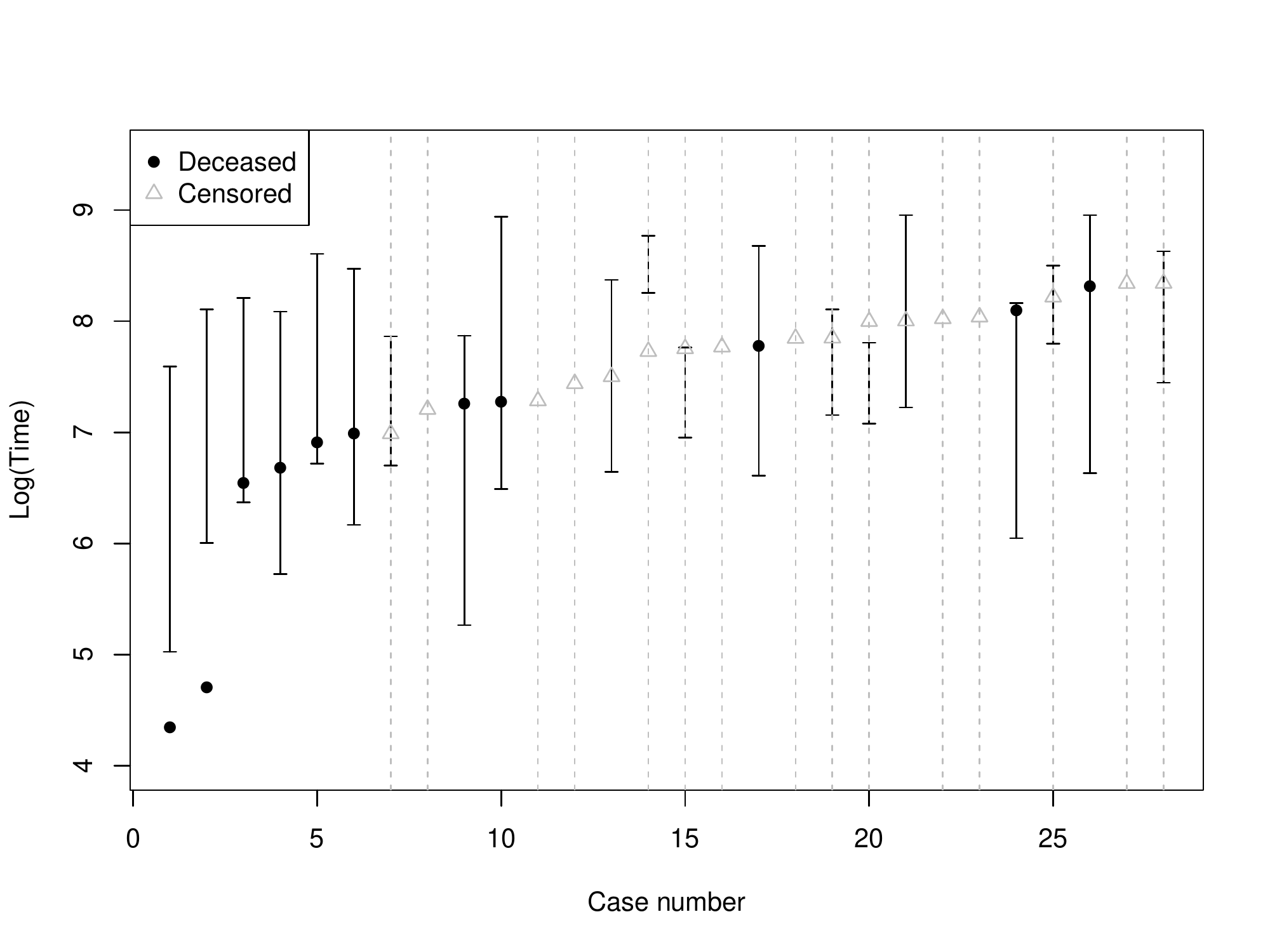}
 	\caption{GCSE prediction intervals for patients in the PBC test set. The black dot and grey triangle respectively represent actual deceased and censored patients. The dotted line indicates patients predicted as censored. Prediction intervals of logarithm of survival times were constructed for uncensored patients.}
 	\label{fig:pbc_ci}
 \end{figure}

 \begin{figure}[H]
	\centering
	\includegraphics[width=5.4 in, height=2.2 in]{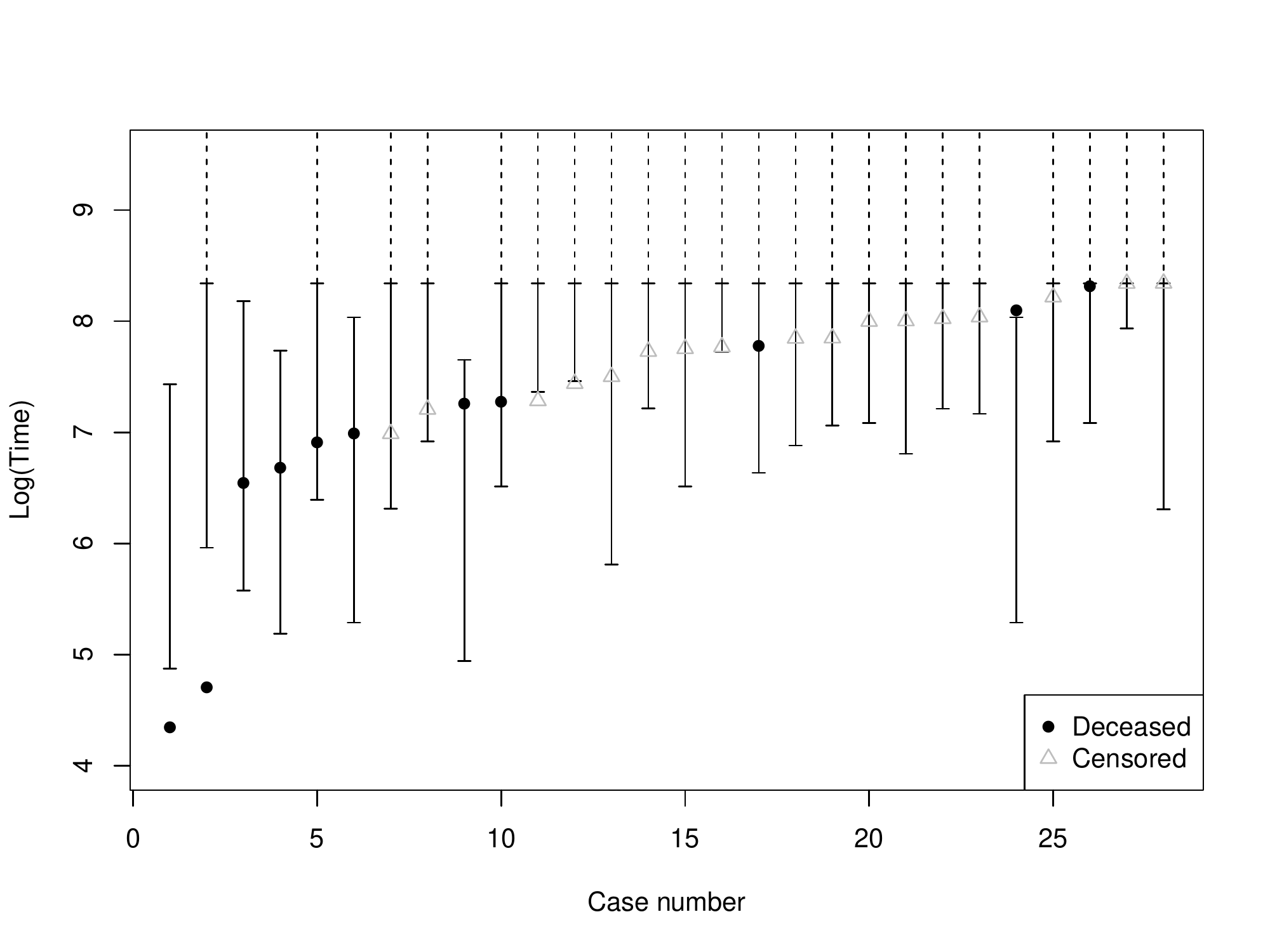}
	\caption{PH prediction intervals for patients in the PBC test set. The black dot and grey triangle respectively represent actual deceased and censored patients. All patients were predicted as censored. Prediction intervals of logarithm of survival times were constructed for uncensored patients.}
	\label{fig:pbc_ci_cox}
\end{figure}


\begin{figure}[H]
	\centering
\includegraphics[width=5.4 in, height=2.2 in]{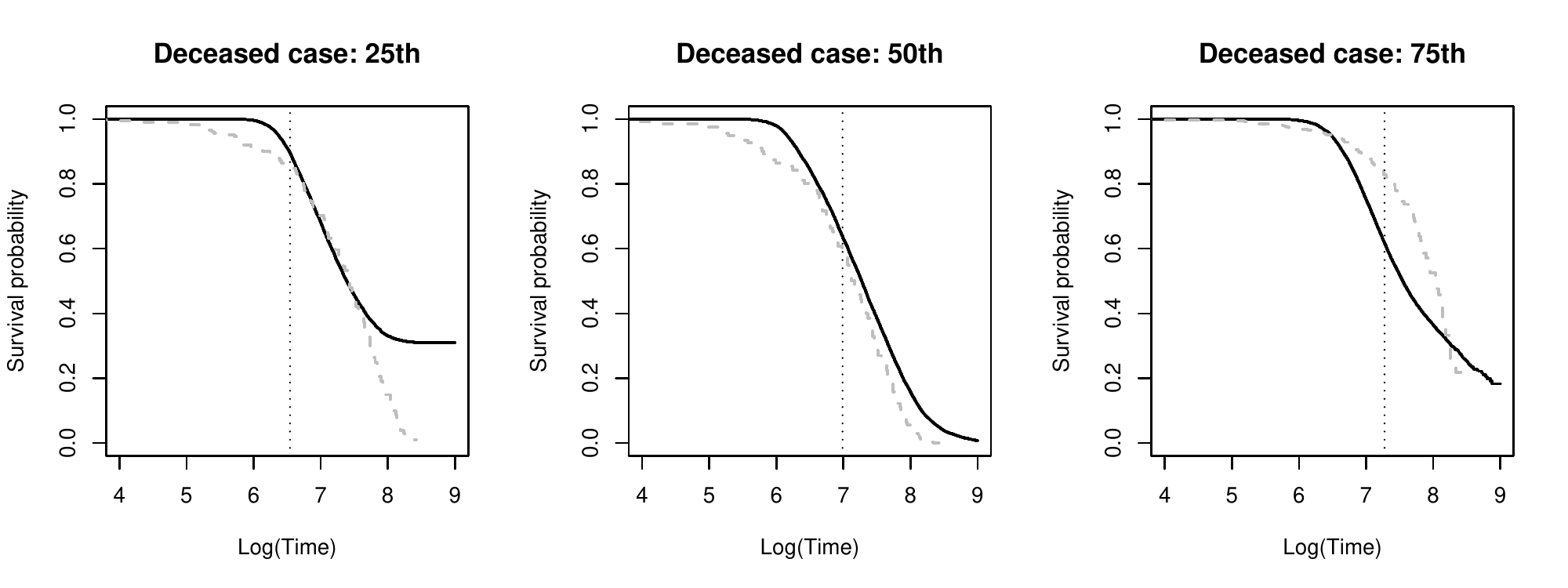}
\caption{Estimated survival function for 3 deceased patients in the PBC test set using GCSE method (black solid line) and PH model (dashed line). The vertical dotted line indicates the actual survival time of the patient.}
	\label{fig:pbc_3de_case}
\end{figure}

\subsubsection{The SUPPORT dataset}
The Study to Understand Prognoses Preferences Outcomes and Risks of Treatment (SUPPORT) dataset contains 9,105 subjects for developing and validating a prognostic model that estimates survival over a 180-day period for hospitalized adult patients diagnosed with seriously illnesses  \citep {knaus1995support}. This dataset is publicly available at \url{https://biostat.app.vumc.org/wiki/Main/DataSets} and includes predictor variables such as diagnosis, age, number of days in the hospital before study entry, presence of cancer, neurologic function, and 11 physiologic measures recorded on day 3 after study entry. Physicians were interviewed on day 3 and patients were followed up for 180 days after enrollment. After removing subjects with missing data and variables with more than 10\% missing data, 7,853 patients with 34 measured predictors remained for analysis. The dataset included 2,451 censored subjects, resulting a censoring rate 31\%.

We used GCSE to train a nonparametric model for predicting patient survival time.
Quantitative predictor variables were standardized and categorical variables were transformed into binary dummy variables. We implemented feedforward neural networks with 2 hidden layers for the generator and the discriminator, both with  $\{60, 30\}$ nodes.
We considered the Elu activation function with $\alpha = 0.3$ for the hidden layers as well as noise vector $\eta \sim N(\0, \mathbf{I}_{20}).$

\begin{figure}[H]
	\centering
	\includegraphics[width=5.4 in, height=2.2 in]{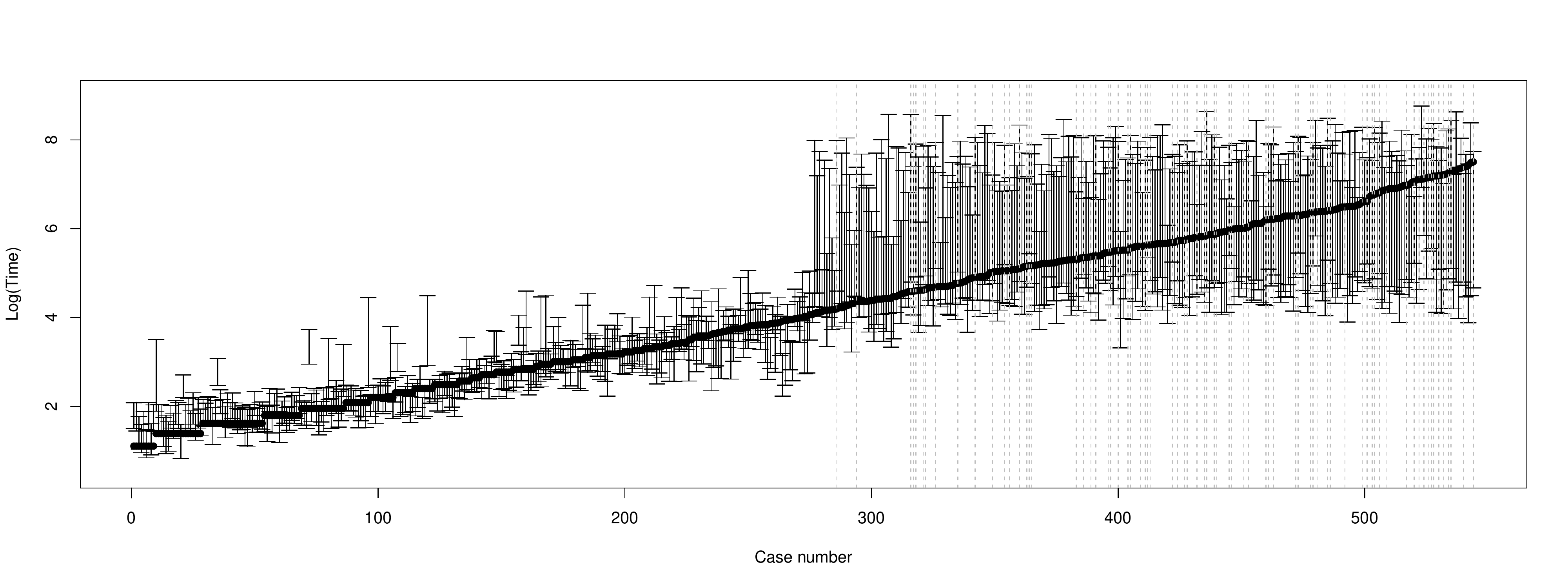}
\caption{The SUPPORT dataset: prediction intervals using GCSE for the survival times of deceased patients in the test set. The gray dashed lines indicate subjects predicted to be censored.}
\label{fig:support_ci_de}
\end{figure}

\begin{figure}[H]
	\centering
	\includegraphics[width=5.4 in, height=2.2 in]{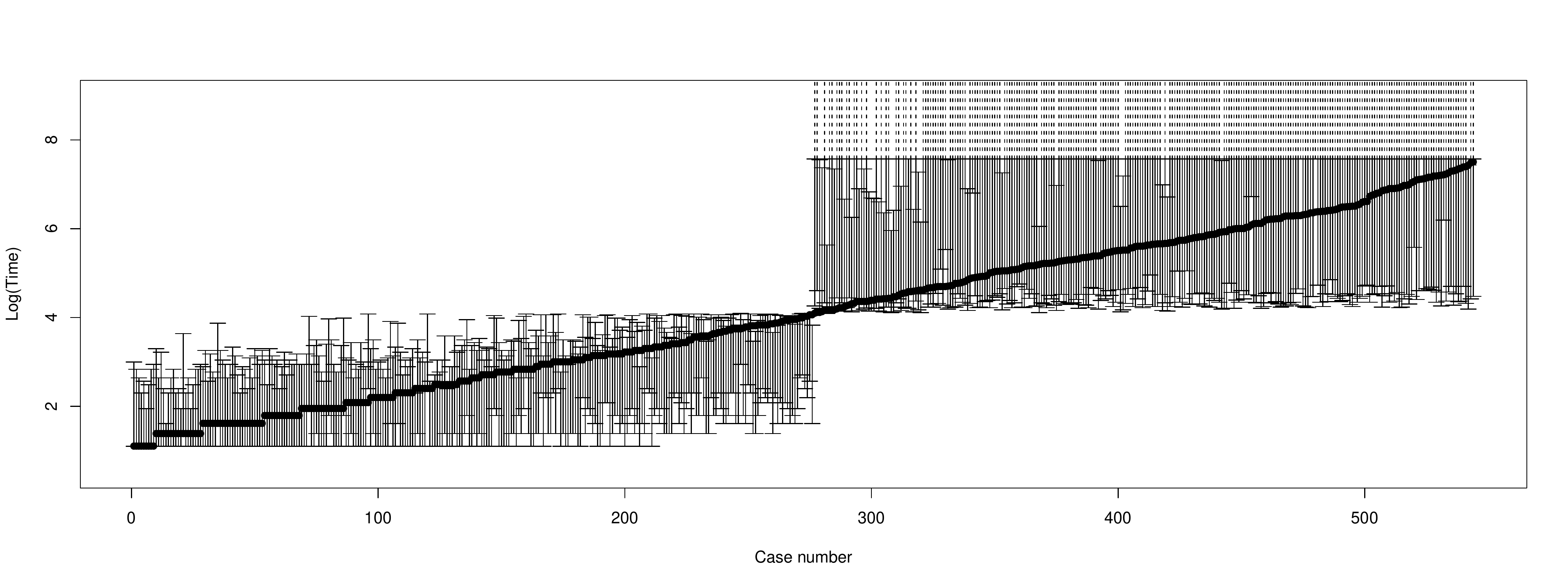}
\caption{SUPPORT data: prediction intervals using PH model for the survival times of deceased patients in the test set.}
\label{fig:support_ci_cox_de}
\end{figure}

Figures \ref{fig:support_ci_de} shows prediction results based on the GCSE method for patients in the test set, with those predicted as censored represented by grey dashed lines and 90\% prediction intervals of survival times constructed for those predicted as uncensored. Patients with predicted censoring probability higher than 50\% were classified as censored. The algorithm generates prediction intervals as long as the predicted censoring probability is less than 100\%, hence we observe some prediction intervals among the censored patients in Figure \ref{fig:support_ci_de}.
For patients with shorter survival times and more severe health conditions, prediction intervals in lower left of the plot are noticeably narrower and thus more reliable. In contrast, patients who survive longer are more likely to be censored and hence resulting much wider prediction intervals.
Similarly, Figure \ref{fig:support_ci_cox_de}
 shows prediction results based on the PH model for patients in the test set. All patients were predicted as uncensored since the PH model cannot make predictions on patient censoring status.
Comparing to the GCSE prediction intervals, those generated by the PH model tend to be wider.

Furthermore, both Figures \ref{fig:support_ci_de} and \ref{fig:support_ci_cox_de} show drastically different prediction intervals above and below some breakpoint at $log({\rm Time})=4.09$, or survival time 60 days. A close examination on the dataset reveals that a particular variable
 describing the status of patient functional disability
exhibits strong predictive power on survival times, such that patients with diminished functional ability tend to live shorter. We also observe some important differences between Figures \ref{fig:support_ci_de} and \ref{fig:support_ci_cox_de}. First, the PH prediction intervals jumped at the breakpoint but remained relatively constant below and above the breakpoint. This is likely due to the proportional hazards assumption as well as the strong influence of the status of patient functional disability. On the other hand, the GCSE prediction intervals wrapped around the point estimate for cases below the break point and also exhibited more variance across the cases. Second, due to censoring,
the estimated survival probabilities based on the PH model for the cases with survival time longer than 60 days ($y>4.09$) located at the upper right part of the plot have a minimum value of 0.3, making it infeasible to construct the 90\% prediction intervals. Instead, we used the longest survival time as the upper bound of the prediction intervals, hence we observe a constant upper bound for patients who survived longer than 60 days.
The estimated survival probabilities based on the GCSE method did not have this issue and were not affected by censoring as severely as the PH model, and hence was possible to construct appropriate 90\% prediction intervals.

\begin{figure}[H]
	\centering
	\includegraphics[width=5.4 in, height=2.2 in]{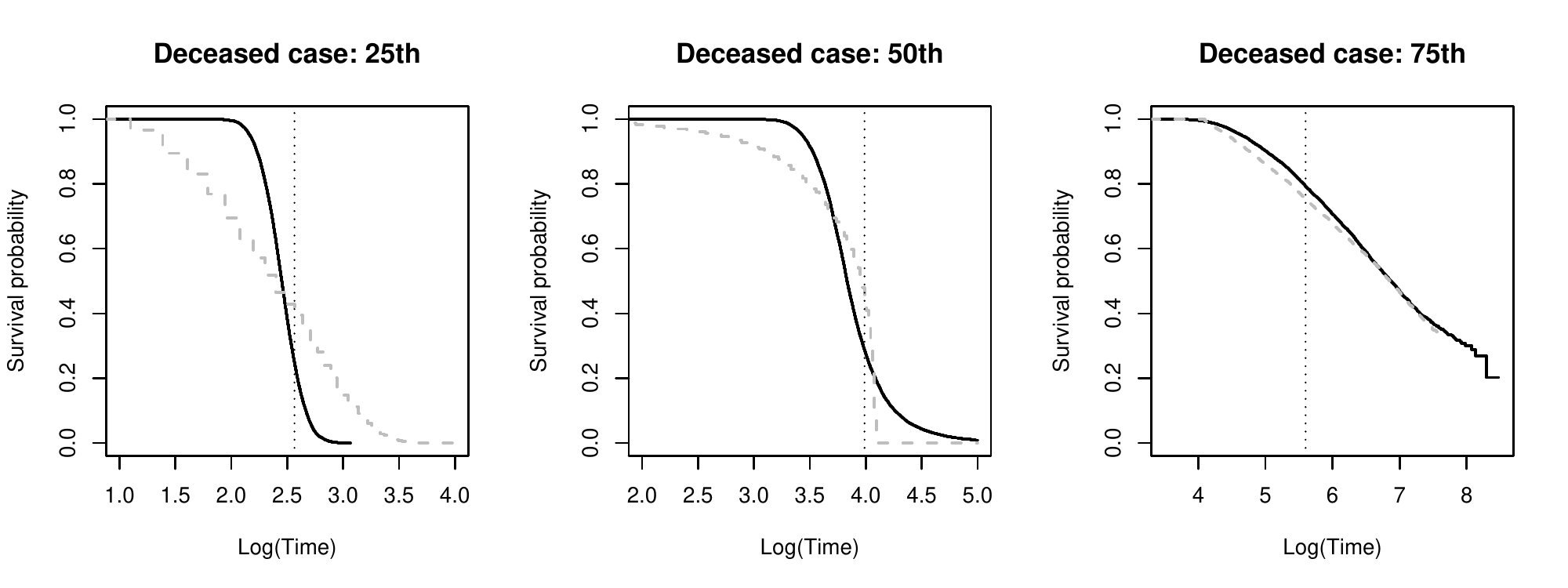}
	\caption{Survival function comparison for 3 deceased cases in SUPPORT testing set. The black solid line is estimation of GCSE. The gray dashed line is estimation by PH. The dotted line is the true time.}
\label{fig:support_3de_case}
\end{figure}

Figure \ref{fig:support_3de_case} compares survival functions estimated by the GCSE method and the PH model for three patients in
the test set. The survival times of these patients are the 25\textsuperscript{th}, 50\textsuperscript{th}, 75\textsuperscript{th} quantiles of the survival times for patients in the test set.

\section{Conclusion}
We proposed the GCSE method, a generative approach for estimating the conditional hazards and survival functions with censored survival data.
Comparing to existing semiparametric regression models for censored survival data, GCSE is nonparametric and model-free since it does not impose any parametric assumptions on how the conditional hazards or survival functions depend on covariates.
The proposed method first estimates a function that can be used to sample from  the conditional distribution of the observed censored response given the covariates, and then for a given value of the covariate,  it generates samples from the estimated conditional generator
and subsequently constructs the Kaplan-Meier and Nelson-Aalen estimators of the conditional hazard and survival functions.

We have focused on estimating the conditional hazards and survival functions for constructing prediction intervals of survival times.
Several problems deserve further study in the future.
For example,
it would be interesting to study ways for estimating treatment effects by including a semiparametric component in the proposed framework. It would also be interesting to consider incorporating additional information such as sparsity and latent low dimensional structure to mitigate the curse of dimensionality in the high-dimensional settings.

\section{Proofs}

In this section we prove the results stated in Section 3.

\subsection{Preliminary lemmas}
We first need the following lemmas.

\begin{lemma}
	\label{lem1a}
	Let  $(L_1, W_1)$ of $D_{\vphi}$ and $(L_2, W_2)$ of $G_{\vtheta}$ be specified such that
	$W_1 L_1=\ceil*{\sqrt{n}}$ and $W_2^2 L_2= c q n $ for some constants $12\leq c\leq 384$.
	Suppose that the conditional generator satisfies $\|G_{\vtheta}\|_{\infty} \le 1+\log n.$
	Then, under  Assumptions \ref{asp2} and \ref{asp3}, we have
	\begin{align*}
	\mathbb{E}_{\hG}d_{\cF_B^1}(P_{X,\hG}, P_{X,Y, \Delta}) \le
	C_0 (d+2)^{{1}/{2}} n^{-{1}/{(d+2)}}\log n,
	\end{align*}
	where $C_0$ is a constant independent of $(n, d)$. Here $\mathbb{E}_{\hG}$ represents the expectation with respect to the randomness in $\hG.$
\end{lemma}
Lemma \ref{lem1a} follows from Theorem 3.1 of \citet{lzjh2021}.

{\color{black}
We now derive an error bound for the  generator
$\tG_{n}=(G_{\hat{\vtheta}_{1n}}, G_{\hat{\vtheta}_{2n}}),$ where
$(\hat{\vtheta}_{1n}, \hat{\vtheta}_{2n})$  is defined in (\ref{2}), that is,
$$
(\hat{\vtheta}_{1n},\hat{\vtheta}_{2n}, \hat{\vphi}_n)=\argmin_{\vtheta_1, \vtheta_2}\argmax_{\vphi}\cL_n(G_{\vtheta_1}, G_{\vtheta_2}, D_{\vphi}).
$$
Also, recall that the conditional generator $\hG_n$
is defined in (\ref{3}),  that is,
$$
\hG_n = (\hG_{1n}, \hG_{2n})=(G_{\hat{\vtheta}_{1n}}, 1\{ G_{\hat{\vtheta}_{2n}}\geq 0.5 \}),
$$
where $1\{\cdot\}$ is an indicator function
}



{\color{black}
\begin{lemma}
	\label{thm1}
	Let  $(L_1, W_1)$ of $D_{\vphi}$ and $(L_2, W_2)$ of $G_{\vtheta}$ be specified such that
	$W_1 L_1=\ceil*{\sqrt{n}}$ and $W_2^2 L_2= c q n $ for a constant $12\leq c\leq 384$.
	Suppose that  Assumptions \ref{asp2} and \ref{asp3} hold and
	the conditional generator satisfies $\|G_{\vtheta}\|_{\infty} \le 1+\log n.$
	Then,  we have
	\begin{equation}
	\label{CondConverge1}
	\mathbb{E}_{\tG_n}d_{\cH_B^1}(P_{\tG_n(\eta,X)}, P_{Y,\Delta})
	\leq C_0 (d+2)^{{1}/{2}} n^{-{1}/{(d+2)}}\log n,
	\end{equation}
	where $\cH_B^1$ is the class of bounded 1-Lipschitz functions. Here $\mathbb{E}_{\tG_n}$ represents the expectation with respect to the randomness in $\tG_n.$
\end{lemma}
}

\begin{proof}[Proof of Lemma \ref{thm1}]
	Write $W=(Y, \Delta)$. Let $\cF_B^1$ be the class of bounded 1-Lipschitz functions on
	$\cX \times \cY \times [0, 1]$.  Let $\mathcal{H}_B^1$ be the  class of bounded 1-Lipschitz functions on
	$\cY \times [0, 1]$.
	Let $\cF_0$ be a  subclass of $\cF_B^1$ defined as
	\[
	\cF_0= \{f: f(x,w)=  h(w):  h\in \mathcal{H}_B^1,  (x, w) \in \cX \times \cY\times [0, 1] \}.
	\]
	Then $\cF_0 \subset \cF_B^1$ and  we have
	\begin{align*}
	\sup_{\cF_B^1} \{\Ebb f(X, G(\eta, X)) - \Ebb f(X, W)\}
	& \ge \sup_{\cF_0} \{\Ebb f(X, G(\eta, X))-\Ebb f(X, W)\} \\
	&= \sup_{\mathcal{H}_B^1} \{\Ebb(h(G(\eta, X)) - \Ebb h(W)\}.
	\end{align*}
	Lemma \ref{thm1} follows from Lemma \ref{lem1a} and this inequality.
\end{proof}

To prove  Theorem \ref{lemm2a}, the following result from \citet[Theorem 2.7.1]{vw1996} is needed.
\begin{lemma}\label{lem:lip-entro}
	Let $\cX$ be a bounded, convex subset of $\real^{m}$ with nonempty interior. There exists a constant $c_{m}$ depending only on $m$ such that
	\[
	\log \mathcal{N}(\varepsilon, \cF^1(\cX),\|\cdot\|_{\infty})\leq c_{m}\lambda(\cX^1)\left(\frac{1}{\varepsilon}\right)^m,
	\]
	for every $\varepsilon>0$, where $ \cF^1(\cX)$ is the uniformly bounded 1-Lipschitz  function  class  defined on $\cX$, and $\lambda(\cX^1)$ is the Lebesgue measure of the set $\{x:\|x-\cX\|<1\}.$
\end{lemma}

\subsection{Proof of Theorem \ref{lemm2a}}

\begin{proof}[Proof of Theorem \ref{lemm2a}]
	For  $h \in \cH_B^1$, let
$$H_n(X)=\Ebb(h(\tG_{n})|X) \ \text{ and } \
H^{\star}(X)=\Ebb(h(Y,\Delta)|X).
$$
Denote by $H_n(x)$ when $H_n(X)$ has $X=x$. Then
	$H_n(x)$ is defined on $\cX$ and has $\|H_n\|_{\operatorname{Lip}}\leq K$ based on $\|\tG_{n}\|_{\operatorname{Lip}}\leq K$, where $K$ can take value   $\sqrt{2}K_1$. By Proposition 11.2.3 of \cite{dudley2018real},  $H_n$ can be extended to $\real^{d}$ while keeping   the same Lipschitz constant and supremum norm.
	Since $H_n(X)$ is uniformly bounded, it follows by Prohorov's theorem that there  exists a subsequence  $\{(X,  \tG_{n_m}, H_{n_m}(X))\}$ converging in distribution to  $(X, Y,\Delta, V)$ for a random variable $V$.
	For convenience, relabel $\{(X, \tG_{n_m}, H_{n_m}(X))\}$  by $\{(X,  \tG_{n}, H_{n}(X))\}$.  For any bounded continuous function $\phi$, we have
	$$\Ebb [(h(Y,\Delta)-V)\phi(V,X)]=\lim_{n\rightarrow 0}\Ebb[(h(\tG_n)-H_{n}(X))\phi(H_{n}(X), X)]=0.$$
	Since   measurable function can be approximated by  continuous function,  it then follows $\Ebb (h(Y,\Delta)|V,X)=V$, for which  $H^{\star}(X)=\Ebb(V|X)$ is derived.

For any  $\nu>0$, there exists a closed ball $O$ such that  $P(X\in O^{c})\leq\nu.$ Let $O^1=\{x\in\real^d:d(x,O)\leq 1\}$, where $d(\cdot, \cdot)$ denotes the metric induced from the Euclidean norm.
	On the set $O^1$, it follows  by Arzela-Ascoli theorem that there exists a subsequence  $\{H_{n_{l}}\}$  converging uniformly to a function $g_\nu$ with  $\|g_\nu\|_{\operatorname{Lip}}\leq K$. Let $q(x)=\max\{1-d(x, O), 0\}$. Then, $1_{O}\leq q(x)\leq 1_{O^1}$  and $\|q(x)\|_{\operatorname{Lip}}\leq1.$ We have
	\begin{equation*}
	\begin{split}
	\Ebb\big[\{V-g_\nu(X)\}^2\big] &=\Ebb\big[\{V-g_\nu(X)\}^21_{O}\big] +\Ebb\big[\{V-g_\nu(X)\}^21_{O^c}\big]
	\\& \leq \Ebb\big[\{V-g_\nu(X)\}^2q(X)\big] +4B^2P\big(X\in O^{c}\big)
	\\&=\lim_{l\rightarrow \infty} \Ebb\big[\{H_{n_{l}}(X)- g_\nu(X)\}^2q(X)\big] +4B^2P\big(X\in O^{c}\big)
	\\&\leq 4B^2\nu.
	\end{split}
	\end{equation*}
	By taking a sequence $\{v_m\}$ converging to $ 0$,  we show that $V$ can be approximated by a sequence of function $\{g_{m}\}$ of $X$ satisfying $\|g_{m}\|_{\operatorname{Lip}}\leq K$.
	This together with $H^{\star}(X)=\Ebb(V|X)$ implies that 
 $V=H^{\star}(X)$. In addition, $H^{\star}(X)$ can be approximated by a sequence of  Lipschitz functions $\{g_{m}\}$, that is,
	\begin{equation}\label{spproxiamte}
	\begin{split}
	\Ebb\big[\{H^{\star}(X)-g_m(X)\}^2\big]  \leq \nu_m.
	\end{split}
	\end{equation}
	%
	
	By Lemma \ref{lem1a}, we can select a sequence $\{c_n\}$ satisfying $r_n/c_n=o(1)$  and $c_n=o(1)$ such that $P_{\tG_n}(d_{\cF_B^1}(P_{X,\tG_n}, P_{X,Y, \Delta})>c_n)=o(1), $ 	where $r_n=  (d+2)^{{1}/{2}} n^{-{1}/{(d+2)}}\log n$. Let $L_n=\{d_{\cF_B^1}(P_{X,\tG_n}, P_{X,Y, \Delta})\leq c_n\},$  we have $P_{\tG_n}(L_n)\rightarrow 1.$

	We first work  on  a given generator $\tG_n\in L_n$. For  $h \in \cH_B^1$, let $\hat{H}_n(X)=\Ebb_{\eta}\big\{h(\tG_n(\eta,X))\big\}$. Denote by $\hat{H}_n(x)$ when $\hat{H}_n(X)$ has $X=x$. We have $\|\hat{H}_n\|_{\infty}\leq B$ and   $\|\hat{H}_n\|_{\operatorname{Lip}}\leq K$
	based on $\|\tG_{n}\|_{\operatorname{Lip}}\leq K$. It  then follows
	\begin{equation}\label{eq:Hn}
	\begin{split}
	\Ebb\big[\hat{H}_n(X)h(\tG_n)-\hat{H}_n(X)h(Y,\Delta) \big] =\Ebb\big[\hat{H}_n(X)^2-\hat{H}_n(X)H^{\star}(X) \big] \leq (K+1)Bc_n.
	\end{split}
	\end{equation}
	And
	\begin{equation*}
	\begin{split}
	\big|\Ebb\big[g_m(X)h(\tG_n)-g_m(X)h(Y,\Delta) \big] \big|=\big|\Ebb\big[g_m(X)\hat{H}_n(X) -g_m(X)H^{\star}(X) \big]\big| \leq (K+1)Bc_n.
	\end{split}
	\end{equation*}
	This together with     \eqref{spproxiamte} shows
	\begin{align*}
	&\Ebb\big[H^{\star}(X)^2 -H^{\star}(X)\hat{H}_n(X)\big]\\
	&=\Ebb\big[(H^{\star}(X)-g_m(X))(H^{\star}(X)-\hat{H}_n(X))  \big] -\Ebb\big[g_m(X)(\hat{H}_n(X)-H^{\star}(X))  \big]
	\\&\leq 2B\Ebb\big[\{H^{\star}(X)-g_m(X)\}^2\big]^{1/2}+(K+1)Bc_n
	\\&\leq 2B\nu_m^{1/2}+(K+1)Bc_n.
	\end{align*}
	Since $\nu_m\rightarrow 0$, it follows that
	\begin{equation}\label{eq:Hstar}
	\Ebb\big[H^{\star}(X)^2 -H^{\star}(X)\hat{H}_n(X)\big]   \leq  (K+1)Bc_n.
	\end{equation}
	Combining \eqref{eq:Hn} and \eqref{eq:Hstar}, we have
	$
	\Ebb\big[\{ \hat{H}_n(X) - H^{\star}(X) \}^2  \big]  \leq 2(K+1)Bc_n.
	$
	Therefore, for $h\in\cH_B^1$, we show
	\begin{equation}\label{eq:probability}
	\begin{split}
	P\big( \big|\hat{H}_n(X)-H^{\star}(X)\big|>\varepsilon \big)\leq \frac{2(K+1)Bc_n}{\varepsilon^2} .
	\end{split}
	\end{equation}
	%
	%
	
	For  a sequence $\{b_n\}$  diverging to infinity, let $K_n=\{|y|\leq b_n\}$ and $K_{1,n}=\{|y|\leq b_n-1\}$.    Assumption \ref{asp2} shows
$$P(Y\in K_n^{c})\leq O\Big(b_n^{-1}\exp(-\frac{b_n^{1+\gamma}}{d+2})\Big)$$
 and
 $$P(Y\in K_{1,n}^{c})\leq O\Big((b_n-1)^{-1}\exp(-\frac{(b_n-1)^{1+\gamma}}{d+2})\Big).$$
 Let $\phi_n(y)=\max\{1-d(y, K_{1,n}),0\}$.
	Then, $1_{K_{1,n}}\leq \phi_n(y)\leq 1_{K_{n}}$ and   $\|\phi_n\|_{\operatorname{Lip}}\leq 1$. Since    $P(G_{\hat{\vtheta}_{1n}}\in K_n) \geq \Ebb \phi_n (G_{\hat{\vtheta}_{1n}})$, $\Ebb \phi_n (Y)\geq P(Y\in K_{1,n})$, and $|\Ebb  \phi_n (G_{\hat{\vtheta}_{1n}})-\Ebb \phi_n (Y)|\leq c_n$, we have
$$ P(G_{\hat{\vtheta}_{1n}}\in K_n^{c})\leq
 O\Big((b_n-1)^{-1}\exp(-\frac{(b_n-1)^{1+\gamma}}{d+2})+ c_n\Big).$$
	Let $B_n=\{G_{\hat{\vtheta}_{1n}}, Y\in  K_n\}$. Then, $P(B_n^c)\leq  P(Y\in K_n^{c})+P(G_{\hat{\vtheta}_{1n}}\in K_n^{c})$.  Lemma \ref{lem:lip-entro} shows that  for  any $h\in\cH_B^1$ constrained on $K_n\times [0,1]$ there exists $j\in\{1,\dots, k_{n,\varepsilon}\}$ such that
$$\|h-h_j\|_{\infty}\leq \varepsilon/3\ \text{ and } \
k_{n,\varepsilon}\leq   \exp(C_1b_n/\varepsilon^2)$$
 for some constant $C_1$. On the set $B_n$, since
 $$\big|\Ebb_{\eta}(h(\tG_{n}(\eta, X))) -\Ebb(h(Y,\Delta)| X) -\{\Ebb_{\eta}(h_j(\tG_{n}(\eta, X)))-\Ebb(h_j(Y,\Delta)| X)\}\big|  \leq 2\varepsilon/3,$$
 we have
	$$ \sup_{h\in\cH_B^1}\Ebb_{\eta}(h(\tG_{n}(\eta, X)))-\Ebb(h(Y,\Delta)| X)  \leq \max_{j=1,\dots, k_{n,\varepsilon}} \Ebb_{\eta}(h_j(\tG_{n}(\eta, X)))-\Ebb(h_j(Y,\Delta)| X)+ \frac{2\varepsilon}{3}. $$
	For each $j$, let $\hat{H}_{n,j}(X)=\Ebb_{\eta}(h_j(\tG_{n}(\eta, X)))$ and $H^{\star}_j(X)=\Ebb(h_j(Y,\Delta)| X)$. By \eqref{eq:probability}, we have
	\begin{equation*}
	\begin{split}
	P\big( \big|\hat{H}_{n,j}(X)-H^{\star}_j(X)\big|>\varepsilon \big)\leq \frac{2(K+1)Bc_n}{\varepsilon^2} .
	\end{split}
	\end{equation*}
	It then follows  that
	\begin{equation}
	\begin{split} \label{eq:decompo}
	&P\big(\sup_{h\in\cH_B^1} \Ebb_{\eta}(h(\tG_{n}(\eta, X)))-\Ebb(h(Y,\Delta)\mid X)> \varepsilon\big)
	\\& \leq P\big(\sup_{h\in\cH_B^1} \Ebb_{\eta}(h(\tG_{n}(\eta, X)))-\Ebb(h(Y,\Delta)\mid X) >\varepsilon\mid B_n) P(B_n)+P(B_n^{c})
	\\& \leq P\big(\max_{j=1,\dots, k_{n,\varepsilon}} \Ebb_{\eta}(h_j(\tG_{n}(\eta, X)))-\Ebb(h_j(Y,\Delta)| X)>\frac{\varepsilon}{3} \mid B_n\big)P(B_n)  +P(B_n^{c})
	\\&\leq \sum_{j=1,\dots, k_{n,\varepsilon}} P\big( \Ebb_{\eta}(h_j(\tG_{n}(\eta, X)))- \Ebb(h_j(Y,\Delta)| X)   >\frac{\varepsilon}{3} \big)  +P(B_n^{c})
	\\&\leq \sum_{j=1,\dots, k_{n,\varepsilon}}P\big( \big|\hat{H}_{n,j}(X)-H^{\star}_j(X)\big|>\frac{\varepsilon}{3}  \big)+P(Y\in K_n^{c})+P(G_{\hat{\vtheta}_{1n}} \in K_n^{c})	
	\\&\lesssim \varepsilon^{-2}\exp\{C_1b_n/\varepsilon^2\} c_n +  (b_n-1)^{-1}\exp(- (b_n-1)^{1+\gamma}/(d+2)).
	\end{split}
	\end{equation}
	Select  $b_n=o(\log c_n^{-1})$ and let $d_n=O(\varepsilon^{-2}\exp\{C_1b_n/\varepsilon^2\} c_n +  b_n^{-1}\exp(- b_n^{1+\gamma}/(d+2)))$. We show there exists a sequence $\{d_n\}$ converging to $0$  such that
	\begin{equation}\label{eq111}
	\begin{split}
	P\big(\sup_{h\in\cH_B^1} \Ebb_{\eta}(h(\tG_{n}(\eta, X)))-\Ebb(h(Y,\Delta)\mid X)> \varepsilon\big)\leq d_n.
	\end{split}
	\end{equation}
	We  now consider the randomness from $\tG_n$. To emphasis that only randomness from $X$ is investigated in \eqref{eq111},  we denote $P$ in \eqref{eq111} by $P_X$ and have
	\begin{equation}\label{eq:expect}
	\begin{split}
	&\Ebb_{\tG_n}\big\{P_X(d_{\cH^1_B}(P_{\tG_n(\eta,X)}, P_{(Y,\Delta)|X})>\varepsilon)\big\}
	\\&=\Ebb_{\tG_n}\big\{P_X\big(\sup_{h\in\cH_B^1} \Ebb_{\eta}(h(\tG_n(\eta, X)))-\Ebb(h(Y,\Delta)| X) > \varepsilon\big) \big\}
	\\&\leq \Ebb_{\tG_n}\big\{P_X\big(\sup_{h\in\cH_B^1} \Ebb_{\eta}(h(\tG_n(\eta, X)))-\Ebb(h(Y,\Delta)| X) > \varepsilon\big) |L_n\big\}+P_{\tG_n}(L_n^{c})
	\\&\leq d_n+\frac{r_n}{c_n}\rightarrow 0,
	\end{split}
	\end{equation}
	from which we demonstrate  $  d_{\cH^1_B}(P_{\tG_n(\eta,X)}, P_{(Y,\Delta)|X})\to_P 0$.  Let $A_{\varepsilon}=\{d_{\cH^1_B}(P_{\tG_n(\eta,X)}, P_{(Y,\Delta)|X})>\varepsilon\}$. It also follows by 	\eqref{eq:decompo} and \eqref{eq:expect}   that
	\begin{equation*}
	\begin{split}
	&\Ebb_{\tG_n}\big\{\Ebb_X(d_{\cH^1_B}(P_{\tG_n(\eta,X)}, P_{(Y,\Delta)|X}) )\big\}
	\\&=\Ebb_{\tG_n}\big\{\Ebb_X(d_{\cH^1_B}(P_{\tG_n(\eta,X)}, P_{(Y,\Delta)|X})1_{A_{\varepsilon}}) \big\}+\Ebb_{\tG_n}\big\{\Ebb_X(d_{\cH^1_B}(P_{\tG_n(\eta,X)}, P_{(Y,\Delta)|X})1_{A_{\varepsilon}^{c}}\big\}
	\\&\leq 2B\Ebb_{\tG_n}\big\{P_X(d_{\cH^1_B}(P_{\tG_n(\eta,X)}, P_{(Y,\Delta)|X})>\varepsilon)\big\}+\varepsilon
	\\&\lesssim \varepsilon^{-2}\exp\{C_1b_n/\varepsilon^2\} c_n +  (b_n-1)^{-1}\exp(- (b_n-1)^{1+\gamma}/(d+2)) +\frac{r_n}{c_n}+\varepsilon.
	\end{split}
	\end{equation*}
	By selecting
	$c_n=r_n\log r_n^{-1}$,  $b_n=\log\{\log r_n^{-1/4}\} $ and  $\varepsilon =\{\log r_n^{-1/4}\}^{-1/2}b_n^{1/2}$, we obtain
	\begin{equation}\label{eqmain: expectation}
	\begin{split}
	\Ebb_{\tG_n}\big\{\Ebb_X(d_{\cH^1_B}(P_{\tG_n(\eta,X)}, P_{(Y,\Delta)|X}) )\big\}
	\lesssim    \{\log r_n^{-1}\}^{-1/2}\log\{\log r_n^{-1}\}^{1/2},
	\end{split}
	\end{equation}
	which completes our proof.
\end{proof}

The next lemma shows the consistency of $\widehat H^x_n$ and $\widehat H^x_{1n}$ defined in (\ref{hHn}).
\begin{lemma}
	\label{lem2a}
	Suppose that $H^x(t)$ and $H^x_1(t)$ are continuous functions of $t$ for every $x\in \cX.$
	Then under the same assumptions and conditions of Theorem \ref{lemm2a}, we have
	\begin{align*}
	\sup_{t \in [0, \tau]} |\widehat H^X_n(t)- H^X(t)| \to_P  0
	\ \text{ and } \sup_{t \in [0, \tau]} |\widehat H^X_{1n}(t)- H^X_1(t)| \to_P  0.
	\end{align*}
\end{lemma}
\begin{proof}[\textit{Proof of Lemma \ref{lem2a}}]
	Let $M_n=d_{\cH^1_B}(P_{\tG_n(\eta,X)}, P_{(Y,\Delta)|X})$. It follows from  Theorem \ref{lemm2a} that every subsequence  $M_{n(m)}$ there is a further subsequence  $M_{n(m_k)}$ that converges  to $0$ almost surely.  Let $\Omega=\{\omega: \lim_{k\rightarrow\infty}M_{n(m_k)}=0\}$. For   $\omega\in\Omega$, denoting $x$ the value of $X(\omega)$,  we have   $\tG_{n(m_k)}(\eta, x)$
	converges in distribution  to  $(Y^x, \Delta^x)$.   Let $f(y,\delta)=(y, 1_{\{\delta\geq0.5\}})$. Since $f$ is continuous on $\real^{+}\times \{0,1\}$, it follows by continuous mapping theorem \citep{vaart2000asymptotic} that $f(\tG_{n(m_k)}(\eta, x))=(\hat{G}_{1n(m_k)}(\eta,x), \hat{G}_{2n(m_k)}(\eta,x))$ converges in distribution to $f(Y^x, \Delta^x)=(Y^x, \Delta^x)$. By the definition of $\widehat H^x_n$ and $\widehat H^x_{1n}$ at (\ref{hHn}), we have, at any continuity point $t$ of $H^x$ and $H^x_1$,
	\begin{align*}
	\widehat H^x_{n(m_k)}(t) \to H^x (t), \ \text{ and }
	\widehat H^x_{1n(m_k)}(t) \to H^x_1(t).
	\end{align*}
	If a sequence of subdistribution functions converges pointwise to a limit that is a continuous function, it also converges uniformly to the limit function (see, e.g.  \citet{vaart2000asymptotic}).
	The uniform convergence follows from the continuity assumption on $H^x$ and $H^x_1$ and the fact that a sequence of subdistribution functions converging pointwise to a continuous subdistribution function also converges uniformly. Therefore, for every subsequence of $\sup_{t \in [0, \tau]} |\widehat H^X_n(t)- H^X(t)|$ or $\sup_{t \in [0, \tau]} |\widehat H^X_{1n}(t)- H^X_1(t) |$, there is a  further subsequence  that converges almost surely to $0$, for which we show  $	\sup_{t \in [0, \tau]} |\widehat H^X_n(t)- H^X(t)| \to_P  0$  and   $\sup_{t \in [0, \tau]} |\widehat H^X_{1n}(t)- H^X_1(t)| \to_P  0.$
\end{proof}

\subsection{Proof of Theorem \ref{thm2}}

Before the proof of Theorem \ref{thm2}, we introduce another result. Let   $\mathcal{S}^{p}\left(z_{0}, \ldots, z_{N+1}\right)$ denote the set of all continuous piecewise linear functions $f: \mathbb{R} \rightarrow \mathbb{R}^{p}$, which have breakpoints only at $z_{0}<z_{1}<\cdots<z_{N}<z_{N+1}$ and are constant on $\left(-\infty, z_{0}\right)$ and $\left(z_{N+1}, \infty\right)$.  The following result is given in \citet[Lemma 3.1]{WangYang2022}.
\begin{lemma}\label{lem:piecewise}
	Suppose that $W \geq 7 p+1, L \geq 2$ and $N \leq(W-p-1)\left\lfloor\frac{W-p-1}{6 p}\right\rfloor\left\lfloor\frac{L}{2}\right\rfloor$. Then for any $z_{0}<z_{1}<\cdots<z_{N}<z_{N+1}$, we have $\mathcal{S}^{p}\left(z_{0}, \ldots, z_{N+1}\right) $ can be represented by a ReLU FNN with width and depth no larger than $W$ and $L$, respectively.
\end{lemma}

\begin{proof}[\textit{Proof of Theorem \ref{thm2}}]
	First, we consider the case in  which  $F^x(t)$  and $Q^x(t)$ are continuous functions of $t$ for every $x\in \cX.$ The consistency statement
	(\ref{HazardConverge}) in Theorem \ref{thm2} follows from the extended continuous mapping theorem
	(Theorem 18.10, \citet{vaart2000asymptotic}) and
	the result that the map
	$
	(H^x, H^x_{1}) \mapsto \Lambda^x
	$
	defined in (\ref{cH2}) is a continuous map from $D[0, \tau] \times D[0, \tau] \to D[0, \tau]$ (actually Hadamard differentiable, see, e.g., \citet{gj1990} or
	Lemma 20.14 in \citet{vaart2000asymptotic}).  Here $D[0, \tau]$ is the space of subdistribution functions on $[0, \tau]$ endowed with the supermum norm.
	
	Similarly, the consistency statement (\ref{SurvivalConverge}) in Theorem \ref{thm2}
follows from the extended continuous mapping theorem and the result that the product integral is a continuous map from $D[0, \tau] \to D[0, \tau]$
	(see, e.g., \citet{gj1990} or
	Lemma 20.14 in \citet{vaart2000asymptotic}).
	
	Then, we extend the result to   general situation in which   $F^x(t)$  and $Q^x(t)$ are arbitrary.     For $x\in\cX$, let $t_{1}^x,t_{2}^x,\dots$ be the set of all points of $t$ where either $F^{x}(t)$ or $Q^{x}(t)$ or both have a jump. Let  $h_{x}(y)=y+\sum_{i:t_{i}^x<y}\frac{1}{i^2}$. We define the function $\lambda$ mapping from $\cX\times \cY\times\{0,1\}\times [0,1]$ to $\cX\times \cY\times\{0,1\}$ as
	\[
	\lambda(x,y,\delta,z)=\begin{cases}
	(x,h_{x}(y),\delta),\quad&  y\ne t_{1}^x,t_{2}^x,\dots,
	\\(x,h_{x}(y)+\frac{1}{2s^2}z,\delta),\quad&  y= t_{s}^x,\delta=1,
	\\(x,h_{x}(y)+\frac{1}{2s^2}(1+z),\delta),\quad&  y= t_{s}^x,\delta=0.
	\end{cases}
	\]
	We also write $\lambda(x,y,\delta,z)=(x, g_x(y,\delta,z),\delta)$.
	Conditional on $X=x$ and generating $\varepsilon$ independently from the uniform distribution on $[0,1]$,   \cite{Major1988}
	showed that $(g_x(Y,\Delta,\varepsilon), \Delta)$ has the same distribution of $(\tilde{Y}^x=\min\{\tilde{T}^x, \tilde{C}^x\},\tilde{\Delta}^x=1(\tilde{T}^x\leq \tilde{C}^x))$ in which $\tilde{T}^x$ and $\tilde{C}^x$ have  conditional distributions $\tilde{ F}^{x}(y)$ and  $ \tilde{Q}^{x}(y)$ that are continuous functions of $y$.   Let $\varepsilon_1,\dots, \varepsilon_{n}$  be a sequence of i.i.d. random variables with the uniform distribution on $[0,1]$ and are independent of the data. We construct a new set  of samples $(\tilde{X}_i, \tilde{Y}_i, \tilde{\Delta}_i) =\lambda(X_i,Y_i,\Delta_i,\varepsilon_i), i=1,\dots, n$. Based on the new samples, we aim to derive a new generator estimator $\bar{G}_n=(\bar{G}_{1n}, \bar{G}_{2n})$.
	The generalized inverse function of  $h_{x}(y)$ is defined as  $h_{x}^{-1}(y)=\inf\{w\in\real: h_{x}(w)\geq y\}$. For every $x$, $h_{x}^{-1}(y)$ is a piecewise linear function. Let $\lambda^{-1}(x,y,\delta)=(x,h_{x}^{-1}(y),\delta)$.  Then we have  $\lambda^{-1}(\lambda(x,y,\delta,z))=(x,y,\delta)$. By the analysis of  \cite{lzjh2021}, with probability tending to 1,   we have   $\sup_{\phi}L_n(\tG_n, D_{\phi})=0$. Under this situation,   we have
	\begin{equation}
	\begin{split}\label{eq:aimeq}
	0&=\sup_{\phi}\big\{ \frac{1}{n}\sum_{i=1}^n D_{\phi}(X_i, \tG_{n}(\eta_i,X_i))
	-\frac{1}{n}\sum_{i=1}^n D_{\phi}(X_i, Y_i, \Delta_i)\big\}
	\\&= \sup_{\phi}\big\{ \frac{1}{n}\sum_{i=1}^n  D_{\phi}(X_i, \tG_{n}(\eta_i,X_i))
	-\frac{1}{n}\sum_{i=1}^n D_{\phi}(\lambda^{-1}\circ\lambda(X_i, Y_i,\Delta_i, \varepsilon_i))\big\}
	\\&= \sup_{\phi}\big\{ \frac{1}{n}\sum_{i=1}^n  D_{\phi}(X_i, G_{\hat{\vtheta}_{1n}}(\eta_i,X_i), G_{\hat{\vtheta}_{2n}}(\eta_i,X_i))
	-\frac{1}{n}\sum_{i=1}^n D_{\phi}(\tilde{X}_i, h_{X_i}^{-1} (\tilde{Y}_i),\tilde{\Delta}_i)\big\}.
	\end{split}
	\end{equation}
	Consider a continuous piecewise linear function $\tilde{g}$ defined by
	$\tilde{g}( G_{\hat{\vtheta}_{1n}}(\eta_i,X_i))=h_{X_i}(G_{\hat{\vtheta}_{1n}}(\eta_i,X_i))$ and $\tilde{g}(h_{X_i}^{-1} (\tilde{Y}_i))=\tilde{Y}_i$ when the two inputs are different,  and $\tilde{g}(h_{X_i}^{-1} (\tilde{Y}_i))=\tilde{Y}_i$ when the two inputs are equal.
	%
	%
	As a result of Lemma \ref{lem:piecewise},  $\tilde{g}$ can be denoted by a ReLU FNN, which together with $t=\sigma(t)-\sigma(-t)$ shows that there is a ReLU FNN $\tilde{f}$ such that $\tilde{f}(x,y,z)=(x,\tilde{g},z)$.
	When the  discriminator class is large enough to represent $\tilde{f}$ and has $\{D_{\tilde{\phi}}\circ\tilde{f}\}\subseteq \{D_{\phi}\}$,  it follows by \eqref{eq:aimeq} and    $X_i=\tilde{X}_i$ that
	\begin{equation}
	\begin{split}\label{eq:relate}
	0&=\sup_{\phi}\big\{ \frac{1}{n}\sum_{i=1}^n D_{\phi}(X_i, \tG_{n}(\eta_i,X_i))
	-\frac{1}{n}\sum_{i=1}^n D_{\phi}(X_i, Y_i, \Delta_i)\big\}
	\\&\geq \sup_{\tilde{\phi}}\big\{ \frac{1}{n}\sum_{i=1}^n D_{\tilde{\phi}}(
	\tilde{X}_i, \tilde{g}\circ  G_{\hat{\vtheta}_{1n}}(\eta_i,\tilde{X}_i),G_{\hat{\vtheta}_{2n}}(\eta_i,\tilde{X}_i))
	-\frac{1}{n}\sum_{i=1}^n D_{\tilde{\phi}}(\tilde{X}_i,  \tilde{Y}_i,\tilde{\Delta}_i)\big\}.
	\end{split}
	\end{equation}
	For each  $i$, we modify the function $h_{X_i}(y)$ to $\tilde{h}_{X_i}(y)$ such that for points $y= h_{X_i}^{-1} (\tilde{Y}_i)$  we have $\tilde{h}_{X_i}(y)=\tilde{g}(y)$. Then, the modification is only at jump point of $h_{X_i}(y)$, and we have
	$ h_{X_i}(y)\leq \tilde{h}_{X_i}(y)$. Hence by \eqref{eq:relate}, we have
	\begin{equation*}
	\begin{split}
	0&=\sup_{\phi}\big\{ \frac{1}{n}\sum_{i=1}^n D_{\phi}(X_i, \tG_{n}(\eta_i,X_i))
	-\frac{1}{n}\sum_{i=1}^n D_{\phi}(X_i, Y_i, \Delta_i)\big\}
	\\&\geq \sup_{\tilde{\phi}}\big\{ \frac{1}{n}\sum_{i=1}^n D_{\tilde{\phi}}(
	\tilde{X}_i, \tilde{h}_{\tilde{X}_i}\circ G_{\hat{\vtheta}_{1n}}(\eta_i,\tilde{X}_i),G_{\hat{\vtheta}_{2n}}(\eta_i,\tilde{X}_i))
	-\frac{1}{n}\sum_{i=1}^n D_{\tilde{\phi}}(\tilde{X}_i,  \tilde{Y}_i,\tilde{\Delta}_i)\big\},
	\end{split}
	\end{equation*}
	which  shows that when $\bar{G}_n(\eta, X)=(\tilde{h}_{X}\circ  G_{\hat{\vtheta}_{1n}}(\eta,X),G_{\hat{\vtheta}_{2n}}(\eta,X))$  we have $\sup_{\tilde{\phi}}L_n(\bar{G}_n, D_{\tilde{\phi}})=0$ with probability tending to 1. By a similar   analysis  of \cite{lzjh2021}, the result of Lemma \ref{lem1a} holds for $\bar{G}_n$, for which   Theorem \ref{lemm2a} follows.
	Denote  by $\bar{G}_n^x=(\bar{G}_{1n}^x, \bar{G}_{2n}^x)$ the  conditional generator when  given $X=x$.   Let $\tilde{H}_n^x(y)=P(\bar{G}_{1n}^x\leq y)$ and  $\tilde{H}_{1n}^x(y)=P(\bar{G}_{1n}^x\leq y, g(\bar{G}_{2n}^x)=1)$, where $g(\delta)=1_{\{\delta\geq0.5\}}$. We have
	\begin{equation*}
	\begin{split}
	\tilde{H}_n^x(h_x(y))=P(\bar{G}_{1n}^x\leq h_x(y))&=P(\tilde{h}_{x}\circ  G_{\hat{\vtheta}_{1n}}(\eta,x) \leq h_x(y)).
	\end{split}
	\end{equation*}
	When $y$ is a continuous point of $h_x$, we have $\tilde{H}_n^x(h_x(y)) =P(G_{\hat{\vtheta}_{1n}}(\eta,x)\leq y)=\hat{H}_n^x(y).$
	When $y$ is a jump point of $h_x$, we have  $\tilde{H}_{n}^x(h_{x+}(y))=P(G_{\hat{\vtheta}_{1n}}(\eta,x)\leq  y)=\hat{H}_{n}^x(y)$.
	Similarly,  when $y$ is a continuous point of $h_x$, we have $\tilde{H}_{1n}^x(h_x(y)) = \hat{H}_{1n}^x(y).$
	When $y$ is a jump point of $h_x$, we have   $\tilde{H}_{1n}^x(h_{x+}(y)) =\hat{H}_{1n}^x(y).$ In addition,    when  $y$ is a continuous point of $h_{x}$, we have  $\tilde{H}^x(h_{x}(y))=H^x(y)$  and $\tilde{H}_{1}^x(h_{x}(y))=H_{1}^x(y)$, and when  $y$ is a jump point of $h_{x}$, we have   $\tilde{H}^x(h_{x+}(y))=H^x(y)$  and $\tilde{H}_{1}^x(h_{x+}(y))=H_{1}^x(y)$.
	As the continuity condition is satisfied by the new samples and the result of Theorem \ref{lemm2a} holds for the new generator estimator $\tilde{G}_n$,    the results of Lemma \ref{lem2a} hold for  $|\tilde{H}_n^X(y)-\tilde{H}^X(y)|$ and $|\tilde{H}_{1n}^X(y)-\tilde{H}_{1}^X(y)|$.
	It then follows that  the results of Lemma \ref{lem2a}   hold for $|\hat{H}_n^X(y)-H^X(y)|$ and $|\hat{H}_{1n}^X(y)-H_{1}^X(y)|$. Therefore,  the results of $\eqref{HazardConverge}$
	and $\eqref{SurvivalConverge}$ in Theorem \ref{thm2} follow  based on the above analysis of the continuous case.
\end{proof}

\subsection{Proof of Lemma \ref{lem:KM}  and Theorem \ref{thm_together}}
The proof of  Lemma \ref{lem:KM} uses the  method of \cite{Foldes1981} and \cite{Major1988}
with some modifications.

\begin{proof}[Proof of Lemma \ref{lem:KM}]	
	{\color{black}	Note that when the estimated generator $\tG_n(\eta, X)$ and $X=x$ are given, the randomness is only   from $\eta$.  For the simplicity of notation,  we omit $\eta$ in $P_{\eta}$ and $\Ebb_{\eta}$ in this proof when it does not cause any confusion.}
	First, we consider the case in which    $\hat{H}_{1n}^x(t)$ and $\hat{H}_n^x(t)$  are continuous functions of $t$.
	Define $A_{nm}=\{\max_{1\leq j \leq m}\hat{Y}_{nj}^x>\tau_n^x\}$. Then,
	\[
	P(\bar{A}_{nm})=\hat{H}_{n}^x(\tau_n^x)^m\leq(1-\nu)^m\leq \exp\{-m\nu\},
	\]
	and   $P(A_{nm})\geq 1-\exp\{-m\nu\}>1/2$ when $m\nu>1.$
	For $\epsilon>0$, we choose a partition $0=\eta_0<\eta_1<\cdots<\eta_{Ln^x(\varepsilon)}=\tau_n^x$ such that $ \hat{S}_{n}^x(\eta_{i-1})- \hat{S}_{n}^x(\eta_{i})\leq \varepsilon/2$ for $i=1,\dots, L_{n}^x(\varepsilon)$ and $ L_{n}^x(\varepsilon)\leq 6/\varepsilon.$ Therefore, for any $t\in[0,\tau_n^x]$, there exists $i$ such that $\eta_{i-1}\leq t< \eta_{i}$. Then,
	\begin{align*}
	&\hat{S}_{nm}^x(t)-\hat{S}_n^x(t)\leq \hat{S}_{nm}^x(\eta_{i-1})-\hat{S}_n^x(\eta_{i})
	\leq \hat{S}_{nm}^x(\eta_{i-1})-\hat{S}_n^x(\eta_{i-1})+\varepsilon/2,
	\\&\hat{S}_{nm}^x(t)-\hat{S}_n^x(t)\geq \hat{S}_{nm}^x(\eta_{i})-\hat{S}_n^x(\eta_{i-1})\geq  \hat{S}_{nm}^x(\eta_{i})-\hat{S}_n^x(\eta_{i})-\varepsilon/2,
	\end{align*}
	which  leads to
	\[
	\sup_{t\in[0,\tau^x_n]}|\hat{S}_{nm}^x(t)-\hat{S}_n^x(t)|\leq \max_{i=1,\dots, L_{n}^x(\varepsilon)}|\hat{S}_{nm}^x(\eta_{i})-\hat{S}_n^x(\eta_{i})|+\varepsilon/2.
	\]
	Consequently, we have
	\begin{equation}\label{eq:maineq1}
	\begin{split}
	&P(\sup_{t\in[0,\tau^x_n]}|\hat{S}_{nm}^x(t)-\hat{S}_n^x(t)|>\varepsilon)
	\\&\leq P(\sup_{t\in[0,\tau^x_n]}|\hat{S}_{nm}^x(t)-\hat{S}_n^x(t)|>\varepsilon\mid   A_{nm})+P(\bar{A}_{nm})
	\\&\leq P(\max_{i=1,\dots, L_{n}^x(\varepsilon)}|\hat{S}_{nm}^x(\eta_{i})-\hat{S}_n^x(\eta_{i})|>\varepsilon/2\mid  A_{nm})+\exp\{-m\nu\}
	\\&\leq \sum_{i=1}^{L_{n}^x(\varepsilon)}P( |\hat{S}_{nm}^x(\eta_{i})-\hat{S}_n^x(\eta_{i})|>\varepsilon/2\mid   A_{nm})+\exp\{-m\nu\}.
	\end{split}
	\end{equation}
	For any $t\in[0,\tau^x_n]$, we have
	\begin{equation}\label{eq:maineq}
	\begin{split}
	\log \hat{S}_{nm}^x(t)-\log\hat{S}_n^x(t)&=\sum_{j=1}^{m}1(\hat{Y}^x_{(nj)}\leq t, \hat{\delta}^x_{(nj)}=1)\log (\frac{m-j}{1+m-j})-\log\hat{S}_n^x(t)
	\\&=-\frac{1}{m}\sum_{j=1}^{m}\frac{1(\hat{Y}^x_{(nj)}\leq t, \hat{\delta}^x_{(nj)}=1)}{1-\hat{H}_{n}^x(\hat{Y}^x_{(nj)})} -\log\hat{S}_n^x(t)
	\\&\quad +\sum_{j=1}^{m}1(\hat{Y}^x_{(nj)}\leq t, \hat{\delta}^x_{(nj)}=1)\big\{\log (\frac{m-j}{1+m-j})+\frac{1}{1+m-j}\big\}
	\\&\quad +\frac{1}{m}\sum_{j=1}^{m}1(\hat{Y}^x_{(nj)}\leq t, \hat{\delta}^x_{(nj)}=1)\big\{\frac{1}{1-\hat{H}_{n}^x(\hat{Y}^x_{(nj)})}-\frac{m}{1+m-j}\big\}
	\\&:=R_{nm,1}+R_{nm,2}+R_{nm,3}.
	\end{split}
	\end{equation}
	For $R_{nm,1}$, we have $\sup_{t\in [0,\tau^x_n]}|\frac{1(\hat{Y}^x_{nj}\leq t, \hat{\delta}^x_{nj}=1)}{1-\hat{H}_{n}^x(\hat{Y}^x_{nj})}|\leq \frac{1}{1-\hat{H}_{n}^x(\tau^x_n)}\leq\nu^{-1}$ and
	\[
	\Ebb\Big\{ \frac{1(\hat{Y}^x_{nj}\leq t, \hat{\delta}^x_{nj}=1)}{1-\hat{H}_{n}^x(\hat{Y}^x_{nj})} \Big\}=-\log\hat{S}_n^x(t).
	\]
	It follows by  Hoeffding's lemma that  for any positive constant $\lambda$  we have
	\begin{align*}
	& P( R_{nm,1} >\varepsilon )
	\\\;&\leq \exp\{-\lambda \varepsilon\}\prod_{j=1}^{m}\Ebb\Big[\exp\Big\{-\frac{\lambda}{m}\big\{\frac{1(\hat{Y}^x_{nj}\leq t, \hat{\delta}^x_{nj}=1)}{1-\hat{H}_{n}^x(\hat{Y}^x_{nj})}+\log \hat{S}_n^x(t)\big\}\Big\} \Big]
	\\\;&\leq \exp\{-\lambda \varepsilon\}\prod_{j=1}^{m}\exp\{\frac{2\lambda^2}{m^2\nu^2}\}
	\leq \exp\{- \frac{m\varepsilon^2\nu^2}{8}\},
	\end{align*}
	where the last inequality follows by selecting $\lambda= \varepsilon m \nu^2/4.$   Therefore, we obtain
	\begin{align}\label{eq:Rnm1}
	P( |R_{nm,1}| >\varepsilon  )
	\leq 2\exp\{- \frac{m\varepsilon^2\nu^2}{8}\}.
	\end{align}
	On the set $A_{nm}$, we have
	\begin{align*}
	|R_{nm,2}|&\leq  \sum_{j=1}^{m}1(\hat{Y}^x_{(nj)}\leq \tau_n^x, \hat{\delta}^x_{(nj)}=1) \frac{2}{(1+m-j)^2}
	\\&\leq \sum_{j=1}^{m}1(\hat{Y}^x_{nj}\leq \tau_n^x, \hat{\delta}^x_{nj}=1) \frac{2}{\{1+\sum_{i=1}^{m}1(\hat{Y}^x_{ni}>\hat{Y}^x_{nj})\}^2}
	\\&\leq \frac{2m}{\{1+\sum_{i=1}^{m}1(\hat{Y}^x_{ni}>\tau_n^x)\}^2},
	\end{align*}
	where the first inequality is based on $|\log(1-x)+x|\leq 2x^2$ when $x\in [0,1/2]$.
	Let $B_{nm}=\big\{\big|\sum_{i=1}^{m}1(\hat{Y}^x_{ni}>\tau_n^x)-m(1-\hat{H}_{n}^x(\tau_n^x))\big|\geq m(1-\hat{H}_{n}^x(\tau_n^x)) \hat{H}_{n}^x(\tau_n^x)\big\}$. Based on Bernstein's inequality,
	\begin{align*}
	P\big(B_{nm} \big)&\leq 2\exp\big\{-m(1-\hat{H}_{n}^x(\tau_n^x)) \hat{H}_{n}^x(\tau_n^x)^2/(2+\hat{H}_{n}^x(\tau_n^x))\big\}
	\\&\leq 2\exp\big\{-m\nu  /10\big\}.
	\end{align*}
	Under condition $\varepsilon>2/m\nu^4$,   $ P (\bar{B}_{nm}   )>1/2$. On the complement of  $B_{nm}$, we have
	\[
	\frac{2m}{\{1+\sum_{i=1}^{m}1(\hat{Y}^x_{ni}>\tau_n^x)\}^2}\leq \frac{2m}{(1+m\nu^2)^2}<\varepsilon.
	\]
	Hence,
	\begin{equation}\label{eq:Rnm2}
	\begin{split}
	P\big(|R_{nm,2}|>\varepsilon \big| A_{nm}\big)&	\leq 2P\Big\{\frac{2m}{\{1+\sum_{i=1}^{m}1(\hat{Y}^x_{ni}>\tau_n^x)\}^2}>\varepsilon  \Big\}
	\\&\leq  2P\big(B_{nm}  \big)
	\\& \leq 4\exp\big\{-m\nu  /10\big\}.
	\end{split}
	\end{equation}
	For $R_{nm,3}$, we have
	\[
	|R_{nm,3}|\leq\frac{m}{(1-\hat{H}_n^x(\tau_n^x))(1+\sum_{i=1}^{m}1(\hat{Y}^x_{ ni }>\tau_n^x ))}\sup_{t\in[0,\tau^x_n]}\Big|\frac{1}{m}+\frac{1}{m}\sum_{i=1}^{m}1(\hat{Y}^x_{ ni }>t) +\hat{H}_n^x(t)-1\Big|.
	\]
	On the complement of  $B_{nm}$,
	\[
	\frac{m}{(1-\hat{H}_n^x(\tau_n^x))(1+\sum_{i=1}^{m}1(\hat{Y}^x_{ ni }>\tau_n^x ))}<\frac{1}{\nu^3}.
	\]
	By Corollary 1 of \cite{Massart1990},
	\[
	P\big(\sup_{t\in[0,\tau^x_n]}\big|\frac{1}{m}\sum_{i=1}^{m}1(\hat{Y}^x_{ ni }>t) +\hat{H}_n^x(t)-1\big|>\frac{\varepsilon\nu^3}{2} \big)\leq 2\exp\{-m\varepsilon^2\nu^6/2\},
	\]
	which together with condition $\varepsilon>4/m\nu^3$ leads to
	\begin{equation}\label{eq:Rnm3}
	\begin{split}
	&P\big(|R_{nm,3}|>\varepsilon \big|  A_{nm}\big)
	\\&\;\leq 4\exp\big\{-m\nu  /10\big\}+ 2P\Big\{  \frac{1}{m} >\frac{\varepsilon\nu^3}{2}  \big|   \bar{B}_{nm}\Big\}
	\\&\quad+2P\Big\{\sup_{t\in[0,\tau^x_n]}\Big|\frac{1}{m}\sum_{i=1}^{m}1(\hat{Y}^x_{ ni }>t) +\hat{H}_n^x(t)-1\Big|>\frac{\varepsilon\nu^3}{2} \big|  \bar{B}_{nm}\Big\}
	\\&\; \leq 4\exp\big\{-m\nu  /10\big\}+  8\exp\{-m\varepsilon^2\nu^6/2\}
	\\&\; \leq 12\exp\{-m\varepsilon^2\nu^6/2\}.
	\end{split}
	\end{equation}
	Consequently,    under condition $\varepsilon>6/\nu^4m$, it follows by \eqref{eq:maineq}, \eqref{eq:Rnm1}, \eqref{eq:Rnm2}, and \eqref{eq:Rnm3} that for any $t\in[0,\tau^x_n]$  we have
	\begin{equation}
	\begin{split}
	&P\big(|\log \hat{S}_{nm}^x(t)-\log\hat{S}_n^x(t)|>\varepsilon\mid   A_{nm} )\\&\leq P\big(| R_{nm,1}|>\frac{\varepsilon}{3}\mid   A_{nm} )
	+P\big(| R_{nm,2}|>\frac{\varepsilon}{3}\mid  A_{nm} )+P\big(| R_{nm,3}|>\frac{\varepsilon}{3}\mid  A_{nm} )
	\\&\leq 4\exp\{-  m\varepsilon^2\nu^2/72\}+4\exp\big\{-m\nu  /10\big\}+12\exp\{-m\varepsilon^2\nu^6/18\}
	\\&\leq 20\exp\{-m\varepsilon^2\nu^6/18\}.
	\end{split}
	\end{equation}
	From $|\hat{S}_{nm}^x(t)-\hat{S}_n^x(t)|\leq|\log \hat{S}_{nm}^x(t)-\log\hat{S}_n^x(t)|$, we obtain
	\begin{equation*}
	\begin{split}
	P\big(|\hat{S}_{nm}^x(t)-\hat{S}_n^x(t)|>\varepsilon\mid   A_{nm} ) \leq 20\exp\{-m\varepsilon^2\nu^6/18\},
	\end{split}
	\end{equation*}
	based on which and \eqref{eq:maineq1} we show
	\begin{equation}\label{eq:reslut1}
	\begin{split}
	P(\sup_{t\in[0,\tau^x_n]}|\hat{S}_{nm}^x(t)-\hat{S}_n^x(t)|>\varepsilon )		
	\leq \frac{c_1}{\varepsilon} \exp\{-c_2m\varepsilon^2\nu^6\},
	\end{split}
	\end{equation}
	where $c_1$ and $c_2$ are constants. Choose $\varepsilon=\sqrt{\frac{2\log m}{c_2m\nu^6}}$.  By \eqref{eq:reslut1}, we obtain
	\begin{equation*}
	\begin{split}
	P\Big(\sup_{t\in[0,\tau^x_n]}|\hat{S}_{nm}^x(t)-\hat{S}_n^x(t)|>\sqrt{\frac{2\log m}{c_2m\nu^6}} \Big)		
	\leq  c_1\sqrt{\frac{c_2\nu^6}{2\log m}} \frac{1}{m^{3/2}},
	\end{split}
	\end{equation*}
	and
	\begin{equation*}
	\begin{split}
	\sum_{m=1}^{\infty}	P\Big(\sup_{t\in[0,\tau^x_n]}|\hat{S}_{nm}^x(t)-\hat{S}_n^x(t)|>\sqrt{\frac{2\log m}{c_2m\nu^6}} \Big)		
	\leq  \sqrt{\frac{c_1^2c_2\nu^6}{2}}	\sum_{m=1}^{\infty}  \sqrt{\frac{1}{ \log m}} \frac{1}{m^{3/2}}<\infty.
	\end{split}
	\end{equation*}
	From Borel-Cantelli lemma \citep[Theorem 2.3.1]{Durrett}, we have
	\[
	P\Big(\sup_{t\in[0,\tau^x_n]}|\hat{S}_{nm}^x(t)-\hat{S}_n^x(t)|>\sqrt{\frac{2\log m}{c_2m\nu^6}} \quad i.o.\Big)=0,
	\]
	where i.o. stands for infinitely often. Therefore, we show
	\begin{equation}\label{eq:resultas}
	P\Big(\sup_{t\in[0,\tau^x_n]}|\hat{S}_{nm}^x(t)-\hat{S}_n^x(t)|=O(\sqrt{\frac{ \log m}{ m }})\Big)=1.
	\end{equation}
	
	Then, we extend the result to general case   based on the method of \cite{Major1988}.   In this case,  the estimator  $\hat{S}^x_{nm}(t)$ is defined as the  general form of Kaplan-Meier product limit estimator  \citep[page 1117]{Major1988}. Let $\hat{t}^x_{n1},\hat{t}^x_{n2},\dots$ be the set of all points where either $\hat{H}_{1n}^x(t)$ or $\hat{H}_n^x(t)$ or both have a jump. Define
	\[
	\hat{h}_n^x(u)=u+\sum_{i:\hat{t}^x_{ni}<u}\frac{1}{i^2}, \quad u\in \real.
	\]
	Let $1-\tilde{S}_n^x(t)$ be defined in a  similar way as that  of $\hat{F}(t)$ in \citet[page 1118]{Major1988}.
	Let $\varepsilon_1,\dots, \varepsilon_{m}$  be a sequence of i.i.d. random variables from the uniform distribution on $[0,1]$ and are independent of the data $(\hat{Y}^x_{nj}, \hat{\delta}^x_{nj})$, $j=1,\dots,m$. We consider a new set  of samples $(\tilde{Y}^x_{nj}, \hat{\delta}^x_{nj})$, where for each $j$,
	\[
	\tilde{Y}^x_{nj}=\begin{cases}
	\hat{h}_n^x( \hat{Y}^x_{nj}),  & \text{if } \hat{Y}^x_{nj}\ne \hat{t}^x_{ni}, i=1,2,\dots,\\
	\hat{h}_n^x( \hat{Y}^x_{nj})+\frac{1}{2i^2}\varepsilon_{j},  & \text{if } \hat{Y}^x_{nj}= \hat{t}^x_{ni}, \text{ and } \hat{\delta}^x_{nj}=1, \\
	\hat{h}_n^x( \hat{Y}^x_{nj})+\frac{1}{2i^2}+\frac{1}{2i^2}\varepsilon_{j},  & \text{if } \hat{Y}^x_{nj}= \hat{t}^x_{ni}, \text{ and } \hat{\delta}^x_{nj}=0.
	\end{cases}
	\]
	Denote the estimator based on the new samples  by $\tilde{S}^x_{nm}(t)$. The process $\hat{S}_{nm}^x(t)-\hat{S}_{n}^x(t)$ and $\tilde{S}^x_{nm}(\hat{h}^x_n(t))-\tilde{S}_n^x(\hat{h}^x_n(t))$ coincide.  And as   the  continuity condition is met by $\tilde{S}^x_{nm}(t)-\tilde{S}_n^x(t)$, the result of \eqref{eq:resultas} follows for   $\hat{S}_{nm}^x(t)-\hat{S}_{n}^x(t)$.
	The uniform consistency of $\hat{\Lambda}_{nm}^x$ follows from \eqref{eq:resultas} and   the result that the product integral is a Hadamard differentiable map from $D[0, \tau_n^x] \to D[0, \tau_{n}^x]$ \citep[Lemma 20.14]{vaart2000asymptotic}.
\end{proof}

\begin{proof}[Proof of  Theorem \ref{thm_together}]
Theorem \ref{thm_together} follows from Theorem \ref{thm2} and Lemma \ref{lem:KM}.

\end{proof}

\begin{funding}
Y. Jiao is supported in part by the National Science Foundation of China (No. 11871474) and by the research fund of KLATASDSMOE of China.
 X. Zhao is supported in part by the Research Grant Council
of Hong Kong (15306521)   and The Hong Kong Polytechnic University (P0030124,
P0034285).
\end{funding}

\bibliographystyle{Chicago}
\bibliography{gcse_bib.bib}

\begin{thebibliography}{}

\bibitem[\protect\citeauthoryear{Aalen}{Aalen}{1978}]{aalen1978}
Aalen, O.~O. (1978).
\newblock Nonparametric inference for a family of counting processes.
\newblock {\em The Annals of Statistics\/}~{\em 6\/}(4), 701--726.

\bibitem[\protect\citeauthoryear{Abadi, Barham, Chen, Chen, Davis, Dean, Devin,
  Ghemawat, Irving, Isard, et~al.}{Abadi et~al.}{2016}]{abadi2016tensorflow}
Abadi, M., P.~Barham, J.~Chen, Z.~Chen, A.~Davis, J.~Dean, M.~Devin,
  S.~Ghemawat, G.~Irving, M.~Isard, et~al. (2016).
\newblock Tensorflow: A system for large-scale machine learning.
\newblock In {\em 12th USENIX Symposium on Operating Systems Design and
  Implementation (OSDI 16)}, pp.\  265--283.

\bibitem[\protect\citeauthoryear{Andersen and Gill}{Andersen and
  Gill}{1982}]{ag1982}
Andersen, P.~K. and R.~D. Gill (1982).
\newblock Cox's regression model for counting processes: A large sample study.
\newblock {\em The Annals of Statistics\/}~{\em 10\/}(4), 1100 -- 1120.

\bibitem[\protect\citeauthoryear{Arjovsky, Chintala, and Bottou}{Arjovsky
  et~al.}{2017}]{arjovsky17}
Arjovsky, M., S.~Chintala, and L.~Bottou (2017).
\newblock {W}asserstein generative adversarial networks.
\newblock In {\em ICML}.

\bibitem[\protect\citeauthoryear{Biau, Sangnier, and Tanielian}{Biau
  et~al.}{2021}]{BiauWGAN2021}
Biau, G., M.~Sangnier, and U.~Tanielian (2021).
\newblock Some theoretical insights into wasserstein gans.
\newblock {\em Journal of Machine Learning Research\/}~{\em 22\/}(119), 1--45.

\bibitem[\protect\citeauthoryear{Buckley and James}{Buckley and
  James}{1979}]{bj1979}
Buckley, J. and I.~James (1979, 12).
\newblock {Linear regression with censored data}.
\newblock {\em Biometrika\/}~{\em 66\/}(3), 429--436.

\bibitem[\protect\citeauthoryear{Chapfuwa, Tao, Li, Page, Goldstein, Duke, and
  Henao}{Chapfuwa et~al.}{2018}]{chapfuwa2018}
Chapfuwa, P., C.~Tao, C.~Li, C.~Page, B.~Goldstein, L.~C. Duke, and R.~Henao
  (2018, 10--15 Jul).
\newblock Adversarial time-to-event modeling.
\newblock In J.~Dy and A.~Krause (Eds.), {\em Proceedings of the 35th
  International Conference on Machine Learning}, Volume~80 of {\em Proceedings
  of Machine Learning Research}, pp.\  735--744. PMLR.

\bibitem[\protect\citeauthoryear{Cox and Oakes}{Cox and Oakes}{1984}]{co1984}
Cox, D. and D.~Oakes (1984).
\newblock {\em Analysis of Survival Data}.
\newblock Monographs on statistics and applied probability. Chapman and Hall.

\bibitem[\protect\citeauthoryear{Cox}{Cox}{1972}]{cox1972}
Cox, D.~R. (1972).
\newblock Regression models and life-tables.
\newblock {\em Journal of the Royal Statistical Society: Series B
  (Methodological)\/}~{\em 34\/}(2), 187--220.

\bibitem[\protect\citeauthoryear{Cox}{Cox}{1975}]{cox1975}
Cox, D.~R. (1975).
\newblock {Partial likelihood}.
\newblock {\em Biometrika\/}~{\em 62\/}(2), 269--276.

\bibitem[\protect\citeauthoryear{Dudley}{Dudley}{2018}]{dudley2018real}
Dudley, R.~M. (2018).
\newblock {\em Real Analysis and Probability}.
\newblock CRC Press.

\bibitem[\protect\citeauthoryear{Durrett}{Durrett}{2019}]{Durrett}
Durrett, R. (2019).
\newblock {\em Probability---theory and examples}, Volume~49 of {\em Cambridge
  Series in Statistical and Probabilistic Mathematics}.
\newblock Cambridge University Press, Cambridge.

\bibitem[\protect\citeauthoryear{Fleming and Harrington}{Fleming and
  Harrington}{1991}]{fh1991}
Fleming, T.~R. and D.~P. Harrington (1991).
\newblock {\em {Counting Processes and Survival Analysis}}.
\newblock John Wiley \& Sons, Inc, New York.

\bibitem[\protect\citeauthoryear{F\"{o}ldes and Rejt\"{o}}{F\"{o}ldes and
  Rejt\"{o}}{1981}]{Foldes1981}
F\"{o}ldes, A. and L.~Rejt\"{o} (1981).
\newblock Strong uniform consistency for nonparametric survival curve
  estimators from randomly censored data.
\newblock {\em Ann. Statist.\/}~{\em 9\/}(1), 122--129.

\bibitem[\protect\citeauthoryear{Gill and Johansen}{Gill and
  Johansen}{1990}]{gj1990}
Gill, R.~D. and S.~Johansen (1990).
\newblock A survey of product-integration with a view toward application in
  survival analysis.
\newblock {\em The Annals of Statistics\/}~{\em 18\/}(4), 1501--1555.

\bibitem[\protect\citeauthoryear{Goodfellow, Pouget-Abadie, Mirza, Xu,
  Warde-Farley, Ozair, Courville, and Bengio}{Goodfellow
  et~al.}{2014}]{goodfellow2014generative}
Goodfellow, I., J.~Pouget-Abadie, M.~Mirza, B.~Xu, D.~Warde-Farley, S.~Ozair,
  A.~Courville, and Y.~Bengio (2014).
\newblock Generative adversarial nets.
\newblock In {\em Advances in neural information processing systems}, pp.\
  2672--2680.

\bibitem[\protect\citeauthoryear{Gulrajani, Ahmed, Arjovsky, Dumoulin, and
  Courville}{Gulrajani et~al.}{2017a}]{gulrajani2017improved}
Gulrajani, I., F.~Ahmed, M.~Arjovsky, V.~Dumoulin, and A.~Courville (2017a).
\newblock Improved training of wasserstein gans.
\newblock {\em arXiv preprint arXiv:1704.00028\/}.

\bibitem[\protect\citeauthoryear{Gulrajani, Ahmed, Arjovsky, Dumoulin, and
  Courville}{Gulrajani et~al.}{2017b}]{Gulrajani2017}
Gulrajani, I., F.~Ahmed, M.~Arjovsky, V.~Dumoulin, and A.~C. Courville (2017b).
\newblock Improved training of wasserstein gans.
\newblock {\em Advances in neural information processing systems\/}~{\em 30}.

\bibitem[\protect\citeauthoryear{Kalbfleisch and Prentice}{Kalbfleisch and
  Prentice}{2002}]{kp2002}
Kalbfleisch, J. and R.~Prentice (2002).
\newblock {\em {The Statistical Analysis of Failure Time Data.} 2nd ed}.
\newblock Wiley Series in Probability and Statistics. John Wiley \& Sons, Inc,
  New York.

\bibitem[\protect\citeauthoryear{Kallenberg}{Kallenberg}{2002}]{kall2002}
Kallenberg, O. (2002).
\newblock {\em Foundations of Modern Probability}.
\newblock Springer-Verlag, New York, 2nd edition.

\bibitem[\protect\citeauthoryear{Kaplan and Meier}{Kaplan and
  Meier}{1958}]{kaplan1958}
Kaplan, E.~L. and P.~Meier (1958).
\newblock Nonparametric estimation from incomplete observations.
\newblock {\em Journal of the American Statistical Association\/}~{\em
  53\/}(282), 457--481.

\bibitem[\protect\citeauthoryear{Kingma and Ba}{Kingma and
  Ba}{2015}]{kingma2015adam}
Kingma, D.~P. and J.~Ba (2015).
\newblock Adam: A method for stochastic optimization.
\newblock In {\em Proceedings of the 3rd International Conference on Learning
  Representation}.

\bibitem[\protect\citeauthoryear{Knaus, Harrell, Lynn, Goldman, Phillips,
  Connors, Dawson, Fulkerson, Califf, Desbiens, et~al.}{Knaus
  et~al.}{1995}]{knaus1995support}
Knaus, W.~A., F.~E. Harrell, J.~Lynn, L.~Goldman, R.~S. Phillips, A.~F.
  Connors, N.~V. Dawson, W.~J. Fulkerson, R.~M. Califf, N.~Desbiens, et~al.
  (1995).
\newblock The support prognostic model: Objective estimates of survival for
  seriously ill hospitalized adults.
\newblock {\em Annals of internal medicine\/}~{\em 122\/}(3), 191--203.

\bibitem[\protect\citeauthoryear{Lin and Ying}{Lin and Ying}{1994}]{ly1994}
Lin, D.~Y. and Z.~Ying (1994, 03).
\newblock {Semiparametric analysis of the additive risk model}.
\newblock {\em Biometrika\/}~{\em 81\/}(1), 61--71.

\bibitem[\protect\citeauthoryear{Liu, Zhou, Jiao, and Huang}{Liu
  et~al.}{2021}]{lzjh2021}
Liu, S., X.~Zhou, Y.~Jiao, and J.~Huang (2021).
\newblock Wasserstein generative learning of conditional distribution.
\newblock {\em arXiv preprint arXiv:2110.10277\/}.

\bibitem[\protect\citeauthoryear{Major and Rejt\H{o}}{Major and
  Rejt\H{o}}{1988}]{Major1988}
Major, P. and L.~Rejt\H{o} (1988).
\newblock Strong embedding of the estimator of the distribution function under
  random censorship.
\newblock {\em Ann. Statist.\/}~{\em 16\/}(3), 1113--1132.

\bibitem[\protect\citeauthoryear{Massart}{Massart}{1990}]{Massart1990}
Massart, P. (1990).
\newblock The tight constant in the {D}voretzky-{K}iefer-{W}olfowitz
  inequality.
\newblock {\em Ann. Probab.\/}~{\em 18\/}(3), 1269--1283.

\bibitem[\protect\citeauthoryear{Mckeague and Sasieni}{Mckeague and
  Sasieni}{1994}]{ms1994}
Mckeague, I.~W. and P.~D. Sasieni (1994, 09).
\newblock {A partly parametric additive risk model}.
\newblock {\em Biometrika\/}~{\em 81\/}(3), 501--514.

\bibitem[\protect\citeauthoryear{Mirza and Osindero}{Mirza and
  Osindero}{2014}]{mirza2014cgan}
Mirza, M. and S.~Osindero (2014).
\newblock Conditional generative adversarial nets.
\newblock cite arxiv:1411.1784.

\bibitem[\protect\citeauthoryear{M{\"u}ller}{M{\"u}ller}{1997}]{muller1997}
M{\"u}ller, A. (1997).
\newblock Integral probability metrics and their generating classes of
  functions.
\newblock {\em Advances in Applied Probability\/}, 429--443.

\bibitem[\protect\citeauthoryear{Nelson}{Nelson}{1972}]{nelson1972}
Nelson, W. (1972).
\newblock Theory and applications of hazard plotting for censored failure data.
\newblock {\em Technometrics\/}~{\em 14\/}(4), 945--966.

\bibitem[\protect\citeauthoryear{Tanielian and Biau}{Tanielian and
  Biau}{2021}]{Biau2021}
Tanielian, U. and G.~Biau (2021).
\newblock Approximating lipschitz continuous functions with groupsort neural
  networks.
\newblock In {\em International Conference on Artificial Intelligence and
  Statistics}, pp.\  442--450. PMLR.

\bibitem[\protect\citeauthoryear{Therneau and Grambsch}{Therneau and
  Grambsch}{2000}]{tg2000}
Therneau, T.~M. and P.~M. Grambsch (2000).
\newblock {\em Modeling Survival Data: Extending the Cox Model}.
\newblock New York: Springer-Verlag.

\bibitem[\protect\citeauthoryear{Tsiatis}{Tsiatis}{1990}]{tsiatis1990}
Tsiatis, A.~A. (1990).
\newblock Estimating regression parameters using linear rank tests for censored
  data.
\newblock {\em The Annals of Statistics\/}~{\em 18\/}(1), 354 -- 372.

\bibitem[\protect\citeauthoryear{van~der Vaart}{van~der
  Vaart}{2000}]{vaart2000asymptotic}
van~der Vaart, A. (2000).
\newblock {\em Asymptotic Statistics}.
\newblock Cambridge University Press.

\bibitem[\protect\citeauthoryear{Van~der Vaart and Wellner}{Van~der Vaart and
  Wellner}{1996}]{vw1996}
Van~der Vaart, A.~W. and J.~A. Wellner (1996).
\newblock {\em Weak Convergence and Empirical Processes: with Applications to
  Statistics}.
\newblock Springer, New York.

\bibitem[\protect\citeauthoryear{Villani}{Villani}{2008}]{villani2008optimal2}
Villani, C. (2008).
\newblock {\em Optimal Transport: Old and New}.
\newblock Springer Berlin Heidelberg.

\bibitem[\protect\citeauthoryear{Wei, Ying, and Lin}{Wei
  et~al.}{1990}]{wyl1990}
Wei, L.~J., Z.~Ying, and D.~Y. Lin (1990, 12).
\newblock {Linear regression analysis of censored survival data based on rank
  tests}.
\newblock {\em Biometrika\/}~{\em 77\/}(4), 845--851.

\bibitem[\protect\citeauthoryear{Yang, Li, and Wang}{Yang
  et~al.}{2021}]{WangYang2022}
Yang, Y., Z.~Li, and Y.~Wang (2021).
\newblock On the capacity of deep generative networks for approximating
  distributions.
\newblock {\em arXiv preprint arXiv:2101.12353\/}.

\bibitem[\protect\citeauthoryear{Zhong, Mueller, and Wang}{Zhong
  et~al.}{2021a}]{ZMW2021a}
Zhong, Q., J.~W. Mueller, and J.-L. Wang (2021a).
\newblock Deep extended hazard models for survival analysis.
\newblock In {\em Advances in Neural Information Processing Systems 34
  pre-proceedings}.

\bibitem[\protect\citeauthoryear{Zhong, Mueller, and Wang}{Zhong
  et~al.}{2021b}]{ZMW2021b}
Zhong, Q., J.~W. Mueller, and J.-L. Wang (2021b).
\newblock Deep learning for the partially linear cox model.
\newblock {\em The Annals of Statistics, in press\/}.

\bibitem[\protect\citeauthoryear{Zhou, Jiao, Liu, and Huang}{Zhou
  et~al.}{2021}]{zjlh2021}
Zhou, X., Y.~Jiao, J.~Liu, and J.~Huang (2021).
\newblock A deep generative approach to conditional sampling.
\newblock {\em Journal of the American Statistical Association, in press\/}.

\end{thebibliography}

\newpage
\appendix

\section{Additional analysis results for the PBC and SUPPORT datasets}
Below we include additional figures from the analysis of the PBC and SUPPORT datasets.

\begin{figure}[H]
	\centering
	\includegraphics[width=5.8 in, height=2.4 in]{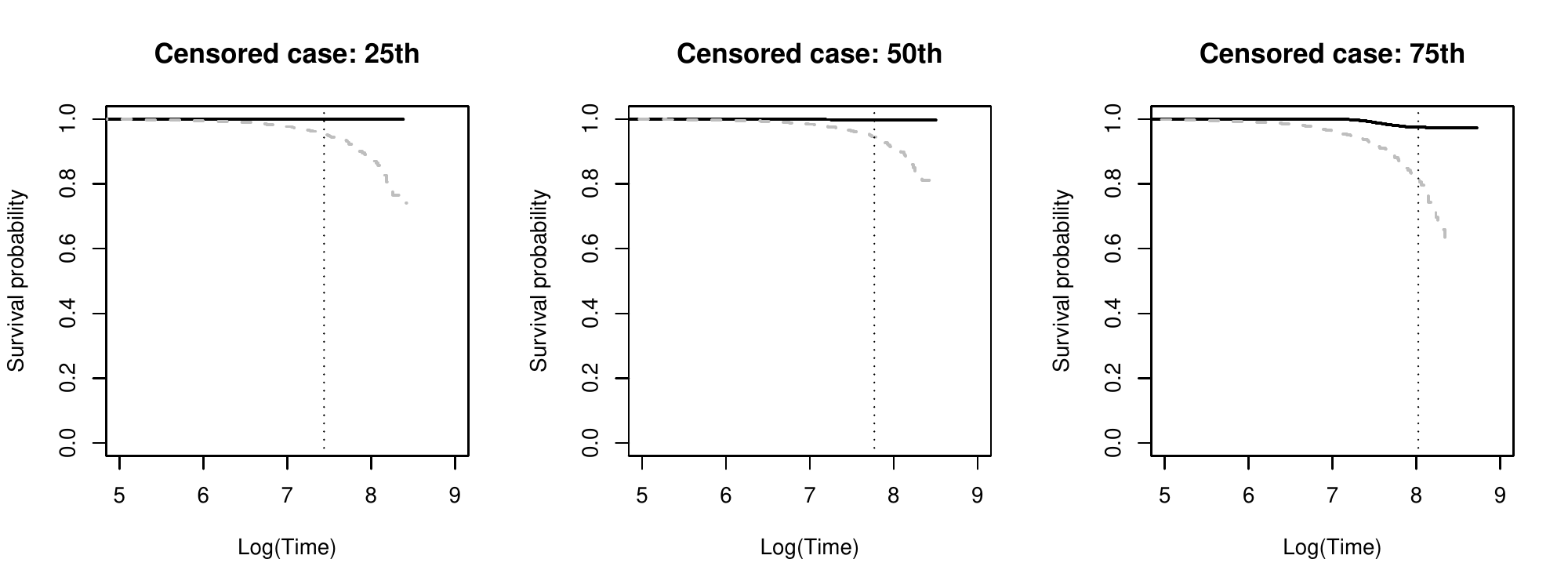}
	\caption{PBC dataset. Estimated survival functions for three censored cases in the test set. The black solid line is the GCSE estimate. The dashed line is the PH estimate. The dotted line indicated the censoring time.}
	\label{fig:pbc_3ce_case}
\end{figure}

In Figure  \ref{fig:pbc_3ce_case}, we plot the survival function estimation for 3 censored patients in test set whose censoring times are the 25th, 50th, 75th quantiles.
For the censored patients, most of the conditional samples generated by our methods are censored hence our survival function is almost a horizontal line and indeed this patient is censored.

\begin{figure}[H]
	\centering
	\includegraphics[width=5.8 in, height=2.4 in]{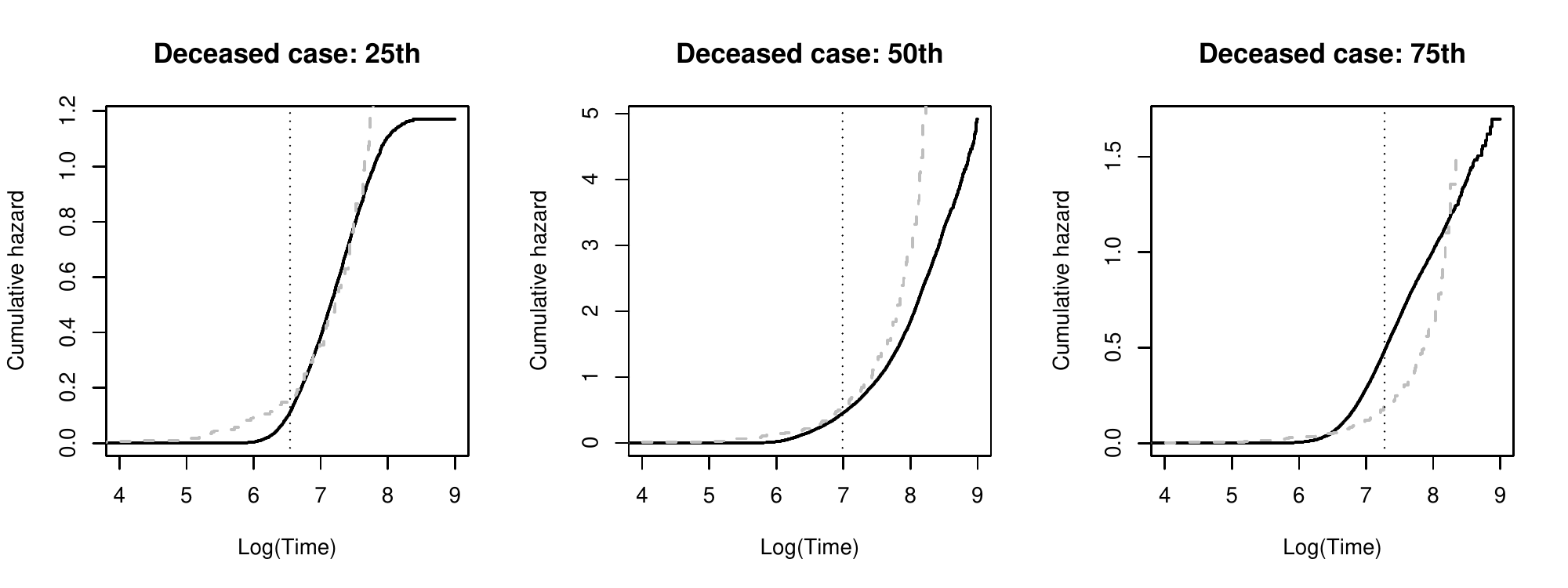}
	\caption{PBC dataset. Estimated cumulative hazard functions for three uncensored cases in the test set. The black solid line is the GCSE estimate. The dashed line is the PH estimate. The dotted line indicated the survival time.}
\label{fig:pbc_3de_case_cumhaz}
\end{figure}

\begin{figure}[H]
	\centering
	\includegraphics[width=5.8 in, height=2.4 in]{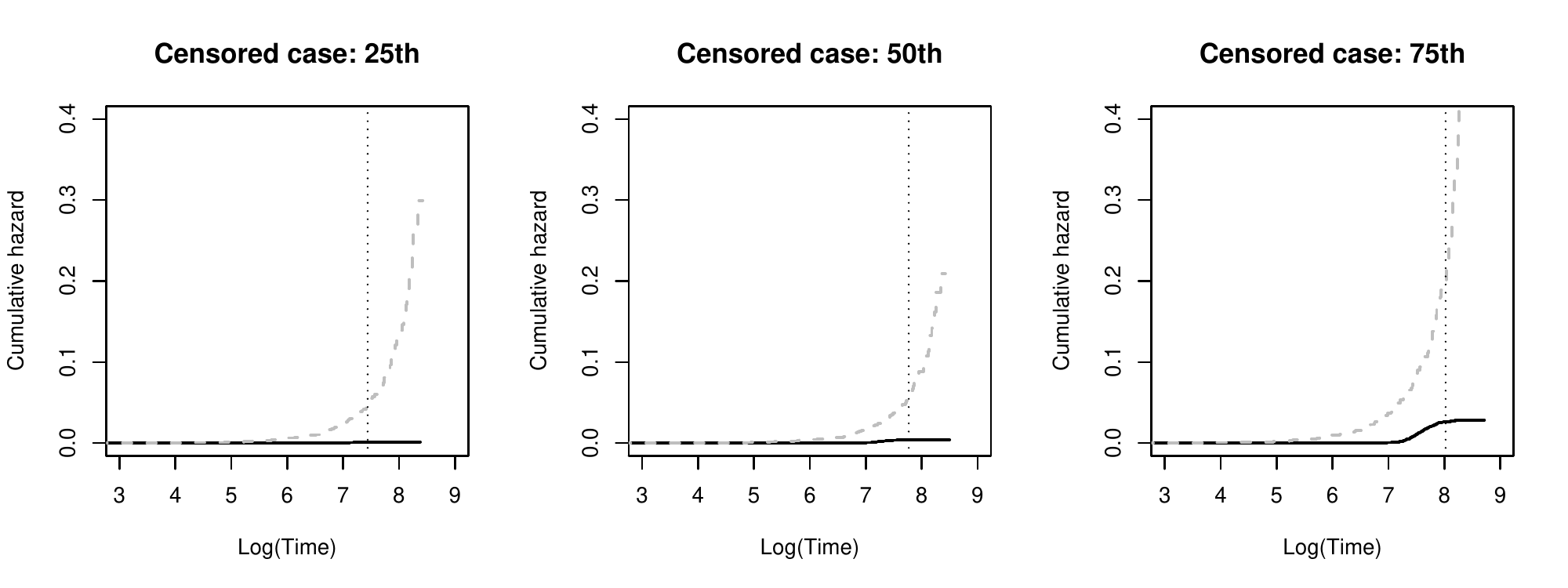}
\caption{PBC dataset. Estimated cumulative hazard functions for three censored cases in the test set. The black solid line is the GCSE estimate. The dashed line is the PH estimate. The dotted line indicated the censoring time.}
	\label{fig:pbc_3ce_case_cumhaz}
\end{figure}

 Figures \ref{fig:pbc_3de_case_cumhaz} and \ref{fig:pbc_3ce_case_cumhaz} show the
 estimated cumulative hazard functions for three patients whose survival times are the 25th, 50th, 75th quantiles of the survival times in the test set.

\begin{figure}[H]
	\centering
	\includegraphics[width=5.8 in, height=2.4 in]{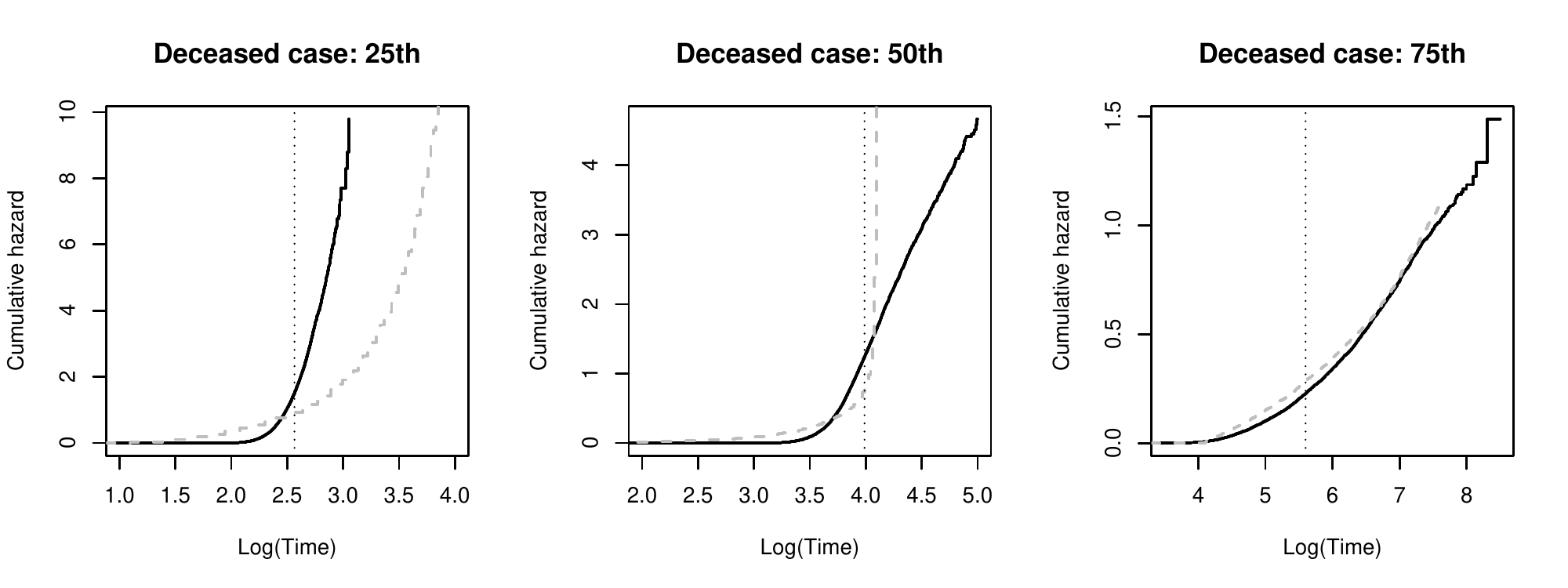}
	\caption{SUPPORT dataset. Estimated cumulative hazard functions for three uncensored cases in the test set. The black solid line is the GCSE estimate. The dashed line is the PH estimate. The dotted line indicated the survival time.}
\label{fig:support_3de_case_cumhaz}
\end{figure}

\begin{figure}[H]
	\centering
	\includegraphics[width=5.8 in, height=2.4 in]{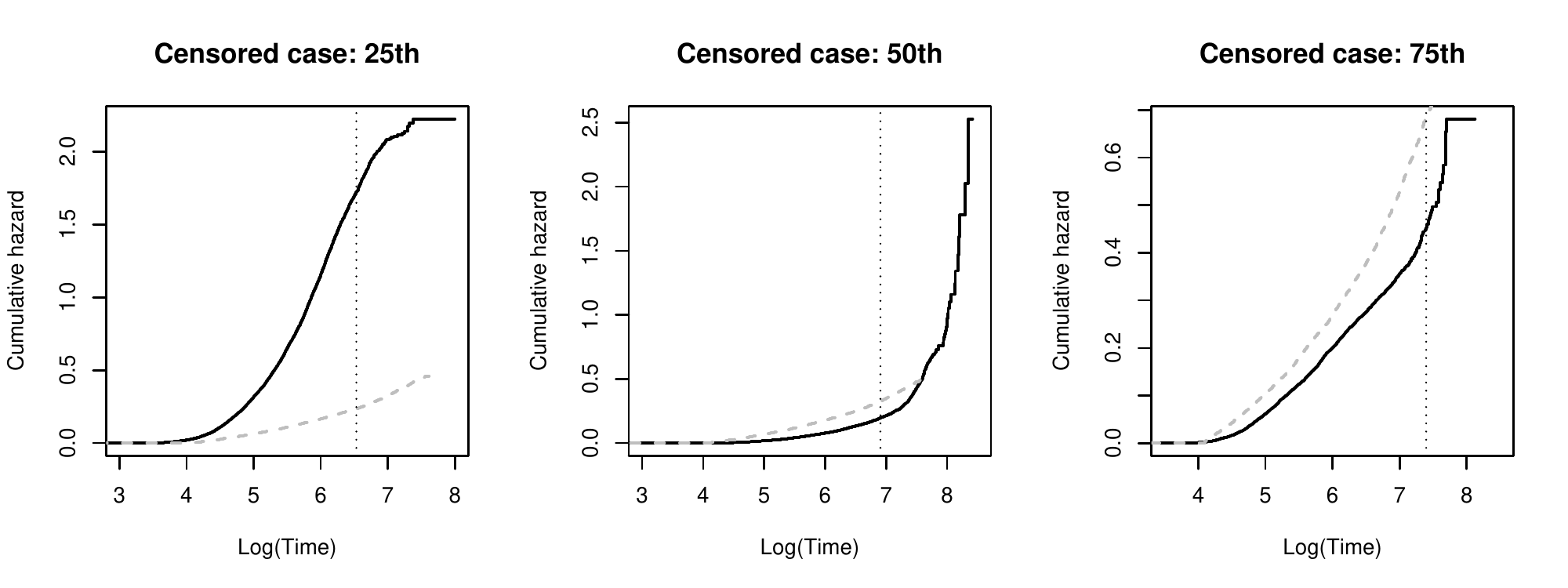}
	\caption{SUPPORT dataset. Estimated cumulative hazard functions for three censored cases in the test set. The black solid line is the GCSE estimate. The dashed line is the PH estimate. The dotted line indicated the censoring time.}
\label{fig:support_3ce_case_cumhaz}
\end{figure}

Figures \ref{fig:support_3de_case_cumhaz} and \ref{fig:support_3ce_case_cumhaz} show the estimated cumulative hazard functions for three patients in the test set.  The three survival times are the 25th, 50th, 75th quantiles of the survival times in the test. Similarly, the three censoring times are the 25th, 50th, 75th quantiles of the censoring times in the test.

\begin{figure}[h]
	\centering
	\includegraphics[width=5.8 in, height=2.4 in]{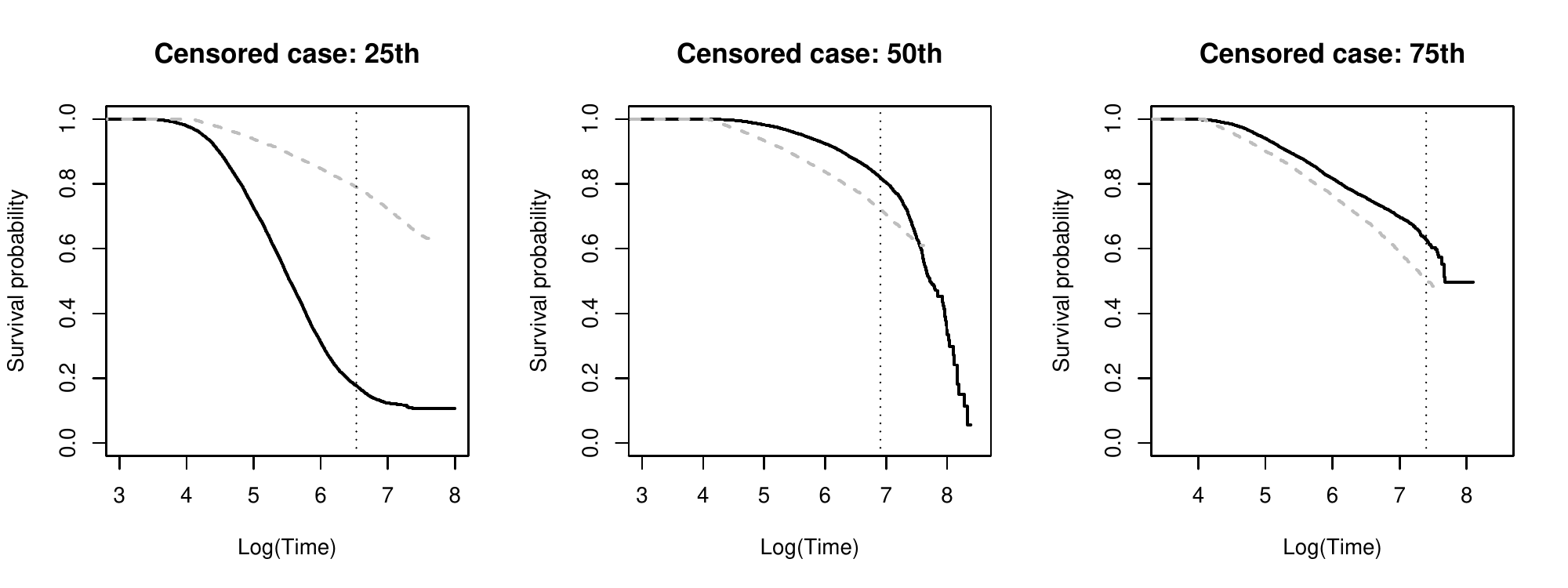}
	\caption{SUPPORT dataset. Estimated survival functions  for three censored cases in the test set. The black solid line is estimation of GCSE. The gray dashed line is estimation by PH. The dotted line is the censoring time.}
\label{fig:support_3ce_case}
\end{figure}

Figure \ref{fig:support_3ce_case} shows the estimated survival function for three censored patients in test set. The three censoring times of these patients are the 25th, 50th, 75th quantiles of the censoring times in the test set.

\end{document}